\documentclass[a4paper]{amsart}
\pdfoutput=1
\usepackage[utf8]{inputenc}
\usepackage[T1]{fontenc}
\usepackage{lmodern}
\usepackage{amssymb,amsxtra}
\usepackage{nicefrac,mathtools,enumitem}

\usepackage{tikz}
\usetikzlibrary{matrix}
\tikzset{cd/.style=matrix of math nodes,row sep=2em,column sep=2em, text height=1.5ex, text depth=0.5ex}
\tikzset{cdar/.style=->,auto}
\tikzset{overar/.style={draw=white,double=black,double distance=.4pt,very thick}}

\newcommand*{\labelar}[5][1]{\draw[cdar] (#2) --
  node[inner sep=#1pt] {\(\scriptstyle #3\)}
  node[inner sep=#1pt,swap] {\(\scriptstyle #4\)} (#5);}

\usepackage{microtype}
\usepackage[pdftitle={Braided multiplicative unitaries as regular objects},
  pdfauthor={Ralf Meyer, Sutanu Roy},
  pdfsubject={Mathematics}
]{hyperref}
\usepackage[lite]{amsrefs}

\renewcommand{\PrintDOI}[1]{\href{http://dx.doi.org/\detokenize{#1}}{doi: \detokenize{#1}}}

\BibSpec{book}{%
  +{}  {\PrintPrimary}                {transition}
  +{,} { \textit}                     {title}
  +{.} { }                            {part}
  +{:} { \textit}                     {subtitle}
  +{,} { \PrintEdition}               {edition}
  +{}  { \PrintEditorsB}              {editor}
  +{,} { \PrintTranslatorsC}          {translator}
  +{,} { \PrintContributions}         {contribution}
  +{,} { }                            {series}
  +{,} { \voltext}                    {volume}
  +{,} { }                            {publisher}
  +{,} { }                            {organization}
  +{,} { }                            {address}
  +{,} { \PrintDateB}                 {date}
  +{,} { }                            {status}
  +{}  { \parenthesize}               {language}
  +{}  { \PrintTranslation}           {translation}
  +{;} { \PrintReprint}               {reprint}
  +{.} { }                            {note}
  +{.} {}                             {transition}
  +{} { \PrintDOI}                   {doi}
  +{} { available at \url}            {eprint}
  +{}  {\SentenceSpace \PrintReviews} {review}
}

\numberwithin{equation}{section}

\theoremstyle{plain}
\newtheorem{theorem}{Theorem}
\numberwithin{theorem}{section}

\newtheorem{lemma}[theorem]{Lemma}
\newtheorem{proposition}[theorem]{Proposition}
\newtheorem{corollary}[theorem]{Corollary}
\newtheorem{assumption}[theorem]{Assumption}

\theoremstyle{definition}
\newtheorem{definition}[theorem]{Definition}
\newtheorem{notation}[theorem]{Notation}

\theoremstyle{remark}
\newtheorem{remark}[theorem]{Remark}
\newtheorem{example}[theorem]{Example}

\newcommand{\tenscorep}{\mathbin{\begin{tikzpicture}[baseline,x=.75ex,y=.75ex] \draw[line width=.2pt] (-0.8,1.15)--(0.8,1.15);\draw[line width=.2pt](0,-0.25)--(0,1.15); \draw[line width=.2pt] (0,0.75) circle [radius = 1];\end{tikzpicture}}}

\newcommand*{\Braiding}[2]{\begin{tikzpicture}[baseline]
    \draw[-] (0,0) -- (1.4ex,1.4ex) node[right,inner sep=0pt] {$\scriptstyle #2$};
    \draw[-,draw=white,line width=2.4pt] (0,1.4ex) -- (1.4ex,0);
    \draw[-] (1.4ex,0) -- (0,1.4ex) node[left,inner sep=0pt] {$\scriptstyle #1$};
  \end{tikzpicture}}
\newcommand*{\Dualbraiding}[2]{\begin{tikzpicture}[baseline]
    \draw[-] (1.4ex,0) -- (0,1.4ex) node[left,inner sep=0pt] {$\scriptstyle #1$};
    \draw[-,draw=white,line width=2.4pt] (0,0) -- (1.4ex,1.4ex);
    \draw[-] (0,0) -- (1.4ex,1.4ex) node[right,inner sep=0pt] {$\scriptstyle #2$};
  \end{tikzpicture}}

\newcommand*{\Corep}[1]{\mathbb{#1}}          %Representation as operator on Hilbert space
\newcommand*{\DuCorep}[1]{\hat{\Corep{#1}}}   %Dual representation viewed as operator on Hilbert space

\newcommand*{\ket}[1]{\lvert#1\rangle}% parts of rank-one operators
\newcommand*{\bra}[1]{\langle#1\rvert}

\newcommand*{\inOb}{\mathrel{\in\in}}% object in a category

\newcommand*{\Cat}{\mathfrak C}     % abstract tensor category
\newcommand*{\Hilb}{\mathfrak{Hilb}}% Hilbert space tensor category
\newcommand*{\Corepcat}[1]{\mathfrak{Rep}(#1)}% representation category of a multiplicative unitary
\newcommand*{\Forget}{\mathfrak{For}}% forgetful functor
\newcommand*{\Trivial}{\tau}% trivial representation functor

\newcommand*{\nb}{\nobreakdash}
\newcommand*{\Star}{$^*$\nb-}

\newcommand*{\C}{\mathbb C}
\newcommand*{\Z}{\mathbb Z}

\newcommand*{\N}{\mathbb N}

\newcommand*{\Comult}[1][]{\Delta_{#1}}%comultiplication

\newcommand*{\Bound}{\mathbb B}%adjointable operators on a Hilbert module
\newcommand*{\congto}{\xrightarrow\sim}%isomorphism to ...

\newcommand*{\univ}{\textup u}%universal
\newcommand*{\Id}{\textup{id}}%identity map

\newcommand*{\Multunit}[1][]{\mathbb W^{#1}}%muliplicative unitary on action on a hilbert space
\newcommand*{\multunit}[1][]{\textup W^{#1}}%multiplicative unitary as a bicharacter of the multiplier algebra
\newcommand*{\DuMultunit}{\widehat{\mathbb W}}%muliplicative unitary on action on a hilbert space
\newcommand*{\ProjBichar}{\mathbb{P}}%projection bicharacter viewed as a unitary acting on Hilber space

\newcommand*{\BrMultunit}{\mathbb F}%braided muliplicative unitary on action on a hilbert space

\newcommand*{\Flip}{\Sigma}% flip operator on Hilbert space
\newcommand*{\Cst}{\textup C^*}% C*-algebra
\newcommand*{\Wst}{\textup W^*}% W*-algebra
\newcommand*{\Hils}[1][H]{\mathcal{#1}}%Hilbert space
\newcommand*{\U}{\mathcal U}%unitary group

\newcommand*{\defeq}{\mathrel{\vcentcolon=}}

\newcommand*{\conj}[1]{\overline{#1}}
\newcommand*{\tworow}[2]{\genfrac{}{}{0pt}{}{#1}{#2}}

\begin{document}

\title[Braided multiplicative unitaries as regular objects]
{Braided multiplicative unitaries\\ as regular objects}

\author{Ralf Meyer}
\email{rmeyer2@uni-goettingen.de}
\address{Mathematisches Institut\\
  Georg-August Universität Göttingen\\
  Bunsenstraße 3--5\\
  37073 Göttingen\\
  Germany}

\author{Sutanu Roy}
\email{sutanu@niser.ac.in}
\address{School of Mathematical Sciences\\
 National Institute of Science Education and Research  Bhubaneswar, HBNI\\
 Jatni, 752050\\
 India}

\begin{abstract}
  We use the theory of regular objects in tensor categories to clarify
  the passage between braided multiplicative unitaries and
  multiplicative unitaries with projection.  The braided
  multiplicative unitary and its semidirect product multiplicative
  unitary have the same Hilbert space representations.
  We also show that the
  multiplicative unitaries associated to two regular objects for the
  same tensor category are equivalent and hence generate isomorphic
  \(\Cst\)\nb-quantum groups.  In particular, a \(\Cst\)\nb-quantum
  group is determined uniquely by its tensor category of
  representations on Hilbert space, and any functor between
  representation categories that does not change the underlying
  Hilbert spaces comes from a morphism of \(\Cst\)\nb-quantum
  groups.
\end{abstract}



\thanks{Dedicated to Professor Sh\^oichir\^o Sakai.  The second author
  was partially supported by an INSPIRE faculty award given by D.S.T.,
  Government of India.}

\maketitle

\section{Introduction}
\label{sec:intro}

The Tannaka--Krein Theorem by Woronowicz~\cite{Woronowicz:Tannaka-Krein}
recovers a compact quantum group from its tensor category of finite-dimensional
representations, together with the forgetful functor to the tensor
category of Hilbert spaces.  We shall prove an analogue of this
result for \(\Cst\)\nb-quantum groups, that is, quantum groups
generated by manageable multiplicative unitaries.  Our result
asserts that an isomorphism between the tensor categories of Hilbert
space representations that does not change the underlying Hilbert
spaces lifts to an isomorphism of the underlying Hopf
\Star{}algebras.  More generally, we shall explain how to extract
multiplicative unitaries from representation categories and how to
lift tensor functors between representation categories to
morphisms of multiplicative unitaries.

This article grew out of a suggestion by David Bücher to clarify the
construction of a semidirect product multiplicative unitary from a
braided multiplicative unitary
in \cites{Meyer-Roy-Woronowicz:Qgrp_proj, Roy:Qgrp_with_proj}.  A
braided multiplicative
unitary is supposed to describe a braided \(\Cst\)\nb-quantum group,
which should be a Yetter--Drinfeld algebra over some other
\(\Cst\)\nb-quantum group, equipped with a comultiplication \(B\to
B\boxtimes B\) into its Yetter--Drinfeld twisted tensor square.  The
semidirect product is constructed
in \cites{Meyer-Roy-Woronowicz:Qgrp_proj, Roy:Qgrp_with_proj} by
writing down a
unitary and checking that it is multiplicative.  The data of a
braided multiplicative unitary consists of four unitaries, subject
to seven conditions.  All four unitaries must appear in the explicit
formula, and all seven conditions must be used in the proof that the
semidirect product is multiplicative.  Thus the direct verification
in~\cite{Meyer-Roy-Woronowicz:Qgrp_proj} is rather complicated.
Here we offer a conceptual explanation for this construction.

The main idea behind this is the theory of regular objects in tensor
categories by Pinzari and Roberts~\cite{Pinzari-Roberts:Regular}.  We
prefer to call them natural right absorbers because the adjective
``regular'' is already used for too many other purposes.  A natural
right absorber in~\(\Cat\) gives rise to a multiplicative
unitary~\(\Multunit\) and a tensor functor from~\(\Cat\) to the
tensor category of Hilbert space representations of~\(\Multunit\).
Representations of the semidirect product multiplicative unitary
should be equivalent to representations of the braided
multiplicative unitary.  This idea already appears in a special case
in~\cite{Kasprzak-Meyer-Roy-Woronowicz:Braided_SU2}.  Here we extend
this result to the general case.  Starting with a braided
multiplicative unitary, we define its representation category and
describe a natural right absorber in it by combining two rather
obvious pieces.  The corresponding multiplicative unitary turns out to
be the semidirect product.
We also show that the functor from representations of the braided
multiplicative unitary to representations of the semidirect
product is an isomorphism of categories.  The most difficult point
here is to prove that any representation of the semidirect product
comes from a representation of the braided multiplicative unitary.

The semidirect product comes with a projection, which is another
multiplicative unitary linked to it by pentagon-like equations.  We
interpret this projection through a projection on the representation
category.  More generally, we show that any tensor functor between
representation categories that does not change the underlying Hilbert
spaces lifts to a morphism between the associated multiplicative
unitaries as defined in \cites{Meyer-Roy-Woronowicz:Homomorphisms,
  Ng:Morph_of_Mult_unit}.  This also implies the weak Tannaka--Krein
Theorem for \(\Cst\)\nb-quantum groups mentioned above.
And it gives yet another equivalent description of quantum groups with
projection.

\section{Natural right absorbers in Hilbert space tensor categories}
\label{sec:absorbers}

We are going to recall the notion of a (right) regular object of a
tensor category from~\cite{Pinzari-Roberts:Regular}.  We call such an
object a natural right absorber, avoiding the overused adjective
``regular''.  Going beyond~\cite{Pinzari-Roberts:Regular}, we show
that different natural
right absorbers give isomorphic multiplicative unitaries with respect
to the morphisms of \(\Cst\)\nb-quantum
groups defined in \cites{Meyer-Roy-Woronowicz:Homomorphisms,
  Ng:Morph_of_Mult_unit}.  We also add a further equivalent
description of such quantum group morphisms through functors between
representation categories, and we show that isomorphic
multiplicative unitaries generate isomorphic \(\Cst\)\nb-quantum
groups.

\begin{notation}
  Let~\(\Hilb\)
  denote the \(\Wst\)\nb-category
  of Hilbert spaces.  This is a symmetric monoidal category for the
  usual tensor product~\(\otimes\)
  of Hilbert spaces, with the obvious associator
  \((\Hils_1\otimes\Hils_2)\otimes\Hils_3 \cong
  \Hils_1\otimes(\Hils_2\otimes\Hils_3)\),
  the obvious unit transformations
  \(\C\otimes\Hils\cong\Hils\cong\Hils\otimes\C\),
  and the obvious symmetric braiding
  \[
  \Sigma\colon \Hils_1\otimes\Hils_2\to \Hils_2\otimes\Hils_1,\qquad
  x_1\otimes x_2\mapsto x_2\otimes x_1.
  \]
\end{notation}

Let~\(\Cat\)
be a \(\Wst\)\nb-category
with a faithful forgetful functor \(\Forget\colon \Cat\to\Hilb\).
Faithfulness allows us to assume that
\(\Cat(x_1,x_2) \subseteq \Bound(\Forget(x_1),\Forget(x_2))\)
for all objects \(x_1,x_2\inOb\Cat\)
(we write~\(\inOb\)
for objects of categories, \(\in\)
for arrows).  We say that \(a\in \Bound(\Forget(x_1),\Forget(x_2))\)
\emph{comes from~\(\Cat\)}
if it belongs to \(\Cat(x_1,x_2)\).
We think of objects in~\(\Cat\)
as Hilbert spaces with some extra structure, such as a
representation of a (braided) multiplicative unitary; the morphisms
are those bounded linear maps that preserve this extra structure.
Motivated by this interpretation, we assume the following throughout
this article:

\begin{assumption}
  \label{assum:identity_arrow_equal}
  If \(\Forget(x)=\Forget(x')\)
  and the identity map on this Hilbert space comes from an arrow
  \(x\to x'\), then \(x=x'\).
\end{assumption}

We also want a functor \(\Trivial\colon \Hilb\to\Cat\)
with \(\Forget\circ\Trivial=\Id_{\Hilb}\).
Thus~\(\Trivial\)
acts as the identity on arrows, and the arrows
\(\Trivial(\Hils_1)\to\Trivial(\Hils_2)\)
in~\(\Cat\)
are exactly all bounded linear operators \(\Hils_1\to\Hils_2\).
We abbreviate \(\Trivial(x) \defeq \Trivial\circ\Forget(x)\)
for \(x\inOb\Cat\).
We interpret~\(\Trivial\)
as the functor that equips a Hilbert space~\(\Hils\)
with the ``trivial'' extra structure to get an object in~\(\Cat\).
The existence of~\(\Trivial\)
is a very weak assumption, which follows, for instance, if~\(\Cat\)
is monoidal and has direct sums.

We assume that~\(\Cat\) is also a monoidal category, but not
necessarily braided, such that both \(\Forget\) and~\(\Trivial\) are
strict monoidal functors.  This means, first, that
\(\Forget(x_1\otimes x_2) = \Forget(x_1)\otimes\Forget(x_2)\) for
all \(x_1,x_2\inOb\Cat\) and \(\Trivial(\Hils_1\otimes\Hils_2) =
\Trivial(\Hils_1)\otimes\Trivial(\Hils_2)\) for all
\(\Hils_1,\Hils_2\inOb\Hilb\).  Secondly, that the tensor unit
in~\(\Cat\) is \(\Trivial(\C)\), which \(\Forget\) maps back to the
tensor unit in~\(\Hilb\).  Thirdly, \(\Forget\) and~\(\Trivial\) map
associators and unit transformations in~\(\Cat\) to the obvious
associators and unit transformations in~\(\Hilb\).
Finally, we require the following assumption, which is trivial to
check in all cases we shall consider:

\begin{assumption}
  \label{assum:arrow_cancellative}
  Let \(x_1,x_2,y\inOb\Cat\)
  and let \(a\colon \Forget(x_1)\to\Forget(x_2)\)
  be such that \(a\otimes \Id\)
  comes from an arrow \(x_1\otimes y \to x_2\otimes y\)
  in~\(\Cat\)
  or \(\Id\otimes a\)
  comes from an arrow \(y\otimes x_1 \to y\otimes x_2\)
  in~\(\Cat\).
  Then~\(a\) itself comes from an arrow \(x_1\to x_2\) in~\(\Cat\).
\end{assumption}

\begin{definition}
  A \emph{Hilbert space tensor category} is a monoidal
  \(\Wst\)\nb-category~\(\Cat\)
  with a faithful, strict monoidal functor
  \(\Forget\colon \Cat\to\Hilb\)
  and a strict monoidal functor
  \(\Trivial\colon \Hilb\to\Cat\)
  satisfying \(\Forget\circ\Trivial=\Id_{\Hilb}\)
  and Assumptions \ref{assum:identity_arrow_equal}
  and~\ref{assum:arrow_cancellative}.
\end{definition}

\begin{example}
  \label{exa:Corepcat}
  Let \(\Multunit\in\U(\Hils\otimes\Hils)\)
  be a multiplicative unitary.  Let~\(\Corepcat{\Multunit}\)
  be the \(\Wst\)\nb-category
  of its (right) Hilbert space representations, with intertwiners as
  arrows.  That is, the objects are pairs \((\Hils[K],\Corep{U})\)
  where~\(\Hils[K]\)
  is a Hilbert space and \(\Corep{U}\in\U(\Hils[K]\otimes\Hils)\)
  satisfies
  \(\Multunit_{23} \Corep{U}_{12} = \Corep{U}_{12}\Corep{U}_{13}
  \Multunit_{23}\)
  in \(\U(\Hils[K]\otimes\Hils\otimes\Hils)\).
  The arrows \((\Hils[K]^1,\Corep{U}^1) \to (\Hils[K]^2,\Corep{U}^2)\)
  are operators \(a\in\Bound(\Hils[K]^1,\Hils[K]^2)\)
  with \(\Corep{U}^2 a_1 = a_1 \Corep{U}^1\),
  where \(a_1 \defeq a\otimes\Id_{\Hils}\) in the leg numbering notation.
  The forgetful functor \(\Corepcat{\Multunit} \to\Hilb\)
  forgets the representation, and
  \(\Trivial(\Hils[K]) \defeq (\Hils[K],1)\).
  The tensor product of two representations
  \(\Corep{U}^i \in \U(\Hils[K]^i \otimes\Hils)\),
  \(i=1,2\),
  is
  \(\Corep{U}^1\tenscorep \Corep{U}^2 \defeq \Corep{U}^1_{13}
  \Corep{U}^2_{23} \in \U(\Hils[K]^1\otimes\Hils[K]^2 \otimes\Hils)\).
  Quick computations show that this is again a representation,
  that~\(\tenscorep\) is associative, and that~\(\Trivial(\C)\)
  is a tensor unit, with the usual associator and unit transformations
  from~\(\Hilb\).
  Since an operator of the form
  \(a_1\in \Bound(\Hils[K]^1\otimes\Hils[K]^2)\)
  for \(a\in \Bound(\Hils[K]^1)\)
  commutes with~\(\Corep{U}^2_{23}\),
  it is an intertwiner for \(\Corep{U}^1_{13} \Corep{U}^2_{23}\)
  if and only if~\(a\)
  is one for~\(\Corep{U}^1\).
  Hence Assumption~\ref{assum:arrow_cancellative} holds.
  Assumption~\ref{assum:identity_arrow_equal} holds because our
  objects are indeed Hilbert spaces with extra structure.
\end{example}

\begin{lemma}
  \label{lem:Sigma_over_trivial}
  Let \(x_1,x_2\inOb\Cat\),
  \(\Hils\inOb\Hilb\).
  Then an operator
  \(a\colon \Forget(x_1)\otimes\Hils \to \Forget(x_2)\)
  comes from an arrow \(\hat{a}\in\Cat(x_1\otimes\Trivial(\Hils),x_2)\)
  if and only if the operators
  \(a_\eta\colon \Forget(x_1)\to \Forget(x_2)\),
  \(\xi\mapsto a(\xi\otimes\eta)\),
  come from arrows in \(\Cat(x_1,x_2)\)
  for all \(\eta\in\Hils\).
  Analogous statements hold for operators
  \(\Hils\otimes\Forget(x_1)\to\Forget(x_2)\),
  \(\Forget(x_1)\to\Forget(x_2)\otimes\Hils\),
  and \(\Forget(x_1)\to\Hils\otimes\Forget(x_2)\).
\end{lemma}

\begin{proof}
  An arrow
  \(\hat{a}\in\Cat(x_1\otimes\Trivial(\Hils),x_2)\)
  gives arrows \(\hat{a}_\eta\)
  in \(\Cat(x_1,x_2)\)
  with \(\Forget(\hat{a}_\eta) = a_\eta\)
  by taking
  \(\hat{a}_\eta \defeq \hat{a}\circ \bigl(\Id_{x_1}\otimes
  \Trivial(\ket{\eta})\bigr)\),
  where \(\ket{\eta}\colon \C\to\Hils\),
  \(\lambda\mapsto \lambda\eta\),
  and where we implicitly identify
  \(x_1\cong x_1\otimes\Trivial(\C)\).
  For the converse, choose an orthonormal
  basis~\((\eta_n)_{n\in\N}\)
  in~\(\Hils\).
  For each \(n\in\N\),
  there is an arrow
  \[
  x_1 \otimes \Trivial(\Hils)
  \xrightarrow{\Id\otimes \Trivial(\bra{\eta_n})}
  x_1 \otimes \Trivial(\C) \cong x_1
  \xrightarrow{\hat{a}_{\eta_n}} x_2.
  \]
  in~\(\Cat\).
  The sum of these operators converges weakly to~\(a\).
  Since \(\Cat(x_1\otimes\Trivial(\Hils),x_2)\)
  is a weakly closed subspace of
  \(\Bound(\Forget(x_1)\otimes \Hils,\Forget(x_2))\),
  it follows that~\(a\)
  comes from~\(\Cat\).
\end{proof}

\begin{remark}
  \label{rem:trivial_unique}
  The functor~\(\Trivial\)
  is unique if it exists.  Let~\(\Hils\)
  be a Hilbert space.  Then any bounded linear operator \(\C\to\Hils\)
  comes from an arrow \(\Trivial(\C)\to\Trivial(\Hils)\)
  in~\(\Cat\).
  Conversely, let \(x\)
  be an object of~\(\Cat\)
  with \(\Forget(x)=\Hils\)
  such that any bounded linear map \(\C\to \Hils\)
  comes from an arrow \(\C\to x\).
  Hence the identity map \(\Trivial(\Hils)\to x\)
  comes from an arrow in~\(\Cat\)
  by Lemma~\ref{lem:Sigma_over_trivial}.  Then \(\Trivial(\Hils) = x\)
  by Assumption~\ref{assum:identity_arrow_equal}.
\end{remark}

Given objects \(x_1,x_2,x_3,y_1,y_2,y_3\inOb\Cat\), there are
canonical maps
\begin{alignat*}{2}
  \Cat(x_1\otimes x_2,y_1\otimes y_2) &\to
  \Cat(x_1\otimes x_2\otimes x_3,y_1\otimes y_2\otimes x_3),
  &\qquad
  T&\mapsto T_{12} = T\otimes \Id_{x_3},\\
  \Cat(x_2\otimes x_3,y_2\otimes y_3) &\to
  \Cat(x_1\otimes x_2\otimes x_3,x_1\otimes y_2\otimes y_3),
  &\qquad
  T&\mapsto T_{23} = \Id_{x_1} \otimes T.
\end{alignat*}
An arrow~\(T_{13}\),
however, cannot always be defined: this would require a
braiding on~\(\Cat\).
Nevertheless, the operator~\(T_{13}\)
may be defined if the object in the middle is of the
form~\(\Trivial(\Hils)\).
Lemma~\ref{lem:Sigma_over_trivial} used twice shows that the flip
operator
\[
\Sigma\colon \Forget(x)\otimes\Hils \to \Hils \otimes \Forget(x),
\qquad \xi\otimes\eta\mapsto \eta\otimes\xi,
\]
comes from an arrow in
\(\Cat(x\otimes\Trivial(\Hils), \Trivial(\Hils) \otimes x)\)
for all \(x\inOb\Cat\),
\(\Hils\inOb\Hilb\).  We use these arrows in~\(\Cat\) to define
\begin{multline*}
  \Cat(x_1\otimes x_2,y_1\otimes y_2) \to
  \Cat(x_1\otimes \Trivial(\Hils)\otimes x_2,y_1\otimes
  \Trivial(\Hils)\otimes y_2),\\
  T\mapsto T_{13} \defeq \Sigma_{23} T_{12} \Sigma_{23}
  = \Sigma_{12} T_{23} \Sigma_{12}.
\end{multline*}

\begin{definition}
  \label{def:natural_absorber}
  Let~\(\Cat\) be a Hilbert space tensor category as above.  A
  \emph{natural right absorber} in~\(\Cat\) is an object
  \(\rho\inOb\Cat\) together with unitaries
  \[
  U^x\colon x\otimes \rho \to \Trivial(x)\otimes\rho\qquad
  \text{for all }x\inOb\Cat
  \]
  with the following properties:
  \begin{enumerate}[label=\textup{(\ref*{def:natural_absorber}.\arabic*)}]
  \item \label{en:natural_absorber1}%
    the unitaries~\(U^x\) are natural, that is, the
    following diagram commutes for any arrow \(a\in\Cat(x_1, x_2)\),
    \(x_1,x_2\inOb\Cat\):
    \[
    \begin{tikzpicture}
      \matrix (m) [cd] {
        x_1\otimes \rho&
        \Trivial(x_1)\otimes\rho\\
        x_2\otimes \rho&
        \Trivial(x_2)\otimes\rho\\
      };
      \labelar{m-1-1}{U^{x_1}}{\cong}{m-1-2};
      \labelar{m-2-1}{U^{x_2}}{\cong}{m-2-2};
      \labelar{m-1-1}{}{a\otimes \Id_\rho}{m-2-1};
      \labelar{m-1-2}{a\otimes \Id_\rho=\Trivial(a)\otimes\Id_\rho}{}{m-2-2};
    \end{tikzpicture}
    \]
  \item \label{en:natural_absorber2}%
    for all \(x_1,x_2\inOb\Cat\), the following diagram of unitaries
    commutes:
    \[
    \begin{tikzpicture}
      \matrix (m) [cd,column sep=3em] {
        x_1\otimes x_2\otimes \rho&
        \Trivial(x_1\otimes x_2)\otimes \rho\\
        x_1\otimes \Trivial(x_2)\otimes \rho&
        \Trivial(x_1)\otimes \Trivial(x_2)\otimes \rho\\
      };
      \labelar{m-1-1}{U^{x_1\otimes x_2}}{}{m-1-2};
      \labelar{m-1-1}{}{U^{x_2}_{23}}{m-2-1};
      \labelar{m-2-1}{U^{x_1}_{13}}{}{m-2-2};
      \draw[double,double equal sign distance] (m-1-2) -- (m-2-2);
    \end{tikzpicture}
    \]
  \end{enumerate}
\end{definition}

\begin{lemma}
  \label{lem:natural_absorber_Trivial}
  If \(\rho\)
  and \((U^x)_{x\inOb\Cat}\)
  are a natural right absorber for~\(\Cat\),
  then \(U^{\Trivial(\Hils)} = \Id_{\Trivial(\Hils)\otimes\rho}\)
  for any Hilbert space~\(\Hils\).
\end{lemma}

\begin{proof}
  Assumption~\ref{en:natural_absorber2} for
  \(x_1=x_2=\C=\Trivial(\C)\)
  implies \(U^{\Trivial(\C)} = \Id_{\C}\).
  Any vector \(\xi\in\Hils\)
  gives an arrow \(\ket{\xi}\colon \Trivial(\C)\to\Trivial(\Hils)\).
  The naturality assumption~\ref{en:natural_absorber1} applied to
  these arrows gives
  \(U^{\Trivial(\Hils)}(\xi\otimes\eta) = \xi\otimes\eta\)
  for all \(\xi\in\Hils\), \(\eta\in\Forget(\rho)\).
\end{proof}

\begin{example}
  \label{exa:Multunit_absorber}
  Let~\(\Multunit\)
  be a multiplicative unitary and let \(\Cat=\Corepcat{\Multunit}\)
  as in Example~\ref{exa:Corepcat}.  The pentagon equation says that
  the unitary~\(\Multunit\)
  is also a representation of itself.  A unitary
  \(\Corep{U}\in \U(\Hils[K]\otimes\Hils)\)
  is a representation if and only if it is an intertwiner
  \[
  (\Hils[K]\otimes\Hils,\Corep{U}_{13}\Multunit_{23})
  = (\Hils[K]\otimes\Hils,\Corep{U}\tenscorep\Multunit)
  \to (\Hils[K]\otimes\Hils, \Id_{\Hils[K]}\tenscorep\Multunit)
  = (\Hils[K]\otimes\Hils, \Multunit_{23}).
  \]
  We claim that~\(\Multunit\)
  with the family of arrows
  \(\Corep{U}\colon (\Hils[K],\Corep{U})\otimes (\Hils,\Multunit) \to
  (\Hils[K],\Id_{\Hils[K]})\otimes (\Hils,\Multunit)\)
  is a natural right absorber in~\(\Corepcat{\Multunit}\).
  First, the arrows in~\(\Corepcat{\Multunit}\)
  are exactly those operators for which the arrows~\(\Corep{U}\)
  above are natural.  Secondly, the tensor product of two
  representations is defined exactly so as to verify
  \ref{en:natural_absorber2}.
\end{example}

\begin{proposition}[\cite{Pinzari-Roberts:Regular}*{Theorem 2.1}]
  \label{pro:absorber}
  Let \((\Cat,\Forget,\Trivial,\otimes)\)
  be a Hilbert space tensor category and let \(\rho\)
  and~\((U^x)_{x\inOb\Cat}\)
  be a natural right absorber for~\(\Cat\).
  For \(x\inOb\Cat\),
  let \(\Hils^x \defeq \Forget(x)\),
  and let us also write \(U^x\)
  for \(\Forget(U^x)\in \U(\Hils^x\otimes\Hils^\rho)\).
  Then~\(U^\rho\)
  is a multiplicative unitary, and~\(U^x\)
  for \(x\inOb\Cat\)
  is a right representation of~\(U^\rho\).
  This construction gives a fully faithful, strict tensor functor
  from~\(\Cat\)
  to the tensor category \(\Corepcat{U^\rho}\)
  of representations of the multiplicative unitary~\(U^\rho\),
  which intertwines the forgetful functors from \(\Cat\) and
  \(\Corepcat{U^\rho}\) to~\(\Hilb\).
\end{proposition}

\begin{proof}
  The condition \ref{en:natural_absorber2} and
  Lemma~\ref{lem:natural_absorber_Trivial} give
  \[
  U^{x\otimes\Trivial(\Hils)} = U^x_{13}\colon
  x\otimes\Trivial(\Hils)\otimes\rho \to
  \Trivial(x)\otimes\Trivial(\Hils)\otimes\rho.
  \]
  Let \(x\inOb\Cat\).
  Then \(U^x\in\Cat(x\otimes\rho,\Trivial(x)\otimes\rho)\)
  is an intertwiner.  So we may apply naturality to it.  This and
  condition \ref{en:natural_absorber2} give the
  commuting diagram of unitaries
  \[
  \begin{tikzpicture}
    \matrix (m) [cd,column sep=3em] {
      \Trivial(x)\otimes \rho\otimes \rho&
      x\otimes \rho\otimes \rho&
      x\otimes \Trivial(\rho)\otimes \rho\\
      \Trivial(x)\otimes \Trivial(\rho)\otimes \rho&
      \Trivial(x\otimes \rho)\otimes \rho&
      \Trivial(x)\otimes \Trivial(\rho)\otimes \rho\\
    };
    \labelar{m-1-2}{}{U^x_{12}}{m-1-1};
    \labelar{m-2-2}{}{U^x_{12}}{m-2-1};
    \labelar{m-1-1}{U^\rho_{23} = U^{\Trivial(x)\otimes\rho}}{}{m-2-1};
    \labelar{m-1-2}{U^{x\otimes \rho}}{}{m-2-2};
    \labelar{m-1-2}{U^\rho_{23}}{}{m-1-3};
    \labelar{m-1-3}{U^x_{13}}{}{m-2-3};
    \draw[double,double equal sign distance] (m-2-2) -- (m-2-3);
  \end{tikzpicture}
  \]
  That is, \(U^x_{12} U^x_{13} U^\rho_{23} = U^\rho_{23} U^x_{12}\).
  When we take \(x=\rho\),
  this is the pentagon equation for~\(U^\rho\).
  For general~\(x\),
  it says that~\(U^x\) is a right representation of~\(U^\rho\).

  The naturality of~\(U^x\)
  says that arrows \(x_1\to x_2\)
  in~\(\Cat\)
  are intertwiners \(U^{x_1}\to U^{x_2}\).
  To prove that we have a fully faithful functor, we must show the
  converse.  So let \(a\colon \Hils^{x_1}\to\Hils^{x_2}\)
  be an intertwiner \(U^{x_1}\to U^{x_2}\).  Then we get an arrow
  \[
  x_1\otimes \rho \xrightarrow{U^{x_1}}
  \Trivial(x_1)\otimes \rho \xrightarrow{\Trivial(a)\otimes\Id_\rho}
  \Trivial(x_2)\otimes \rho
  \xrightarrow{(U^{x_2})^{-1}}
  x_2\otimes \rho
  \]
  Since~\(a\)
  is an intertwiner, the forgetful functor maps this composite arrow
  to \(a\otimes\Id_{\Hils^\rho}\).
  Since this operator comes from~\(\Cat\),
  Assumption~\ref{assum:arrow_cancellative} ensures that~\(a\)
  also comes from~\(\Cat\).
  Thus any intertwiner comes from an arrow in~\(\Cat\).
  This finishes the proof that the functor from~\(\Cat\)
  to the category of right representations of~\(U^\rho\)
  is fully faithful.  By construction, our functor intertwines the
  forgetful functors to~\(\Hilb\).

  The condition \ref{en:natural_absorber2} says
  exactly that \(U^{x_1\otimes x_2}\)
  is the tensor product representation \(U^{x_1}\otimes U^{x_2}\).
  Since we assumed~\(\Forget\)
  to map the associator and unit transformations in~\(\Cat\)
  to the usual ones in~\(\Hilb\),
  the functor \(x\mapsto U^x\)
  from~\(\Cat\)
  to the representation category of~\(U^\rho\)
  is a strict tensor functor.
\end{proof}

We have not found a ``nice'' characterisation when the functor
\(\Cat\to\Corepcat{U^\rho}\)
is essentially surjective, that is, when every representation
of~\(U^\rho\)
comes from an object of~\(\Cat\).
An artificial example where this is not the case is the subcategory of
\(\Corepcat{U^\rho}\)
consisting of all representations that are either trivial or a direct
sum of subrepresentations of~\(\rho\).
This has all the structure that we require.  And it is also closed
under direct sums and subrepresentations.  If \(\Corepcat{U^\rho}\)
is, say, the category of representations of the group~\(\Z\)
of integers, then the representations given by non-trivial characters
on~\(\Z\) are missing in this subcategory.

\begin{example}
  \label{exa:left_corep}
  Let \(\Multunit\in\U(\Hils\otimes\Hils)\)
  be a multiplicative unitary.  A \emph{left representation}
  of~\(\Multunit\) on a Hilbert space~\(\Hils[K]\) is a unitary
  \(\DuCorep{V}\in\U(\Hils\otimes\Hils[K])\) satisfying
  \[
  \DuCorep{V}_{23}\Multunit_{12}
  = \Multunit_{12} \DuCorep{V}_{13} \DuCorep{V}_{23}
  \in \U(\Hils\otimes\Hils\otimes\Hils[K]).
  \]
  The tensor product of two left representations
  \(\DuCorep{V}^i\in\U(\Hils\otimes\Hils[K]^i)\),
  \(i=1,2\), is the left representation on
  \(\Hils[K]^1\otimes\Hils[K]^2\) defined by
  \[
  \DuCorep{V}^1\tenscorep \DuCorep{V}^2 \defeq
  \DuCorep{V}^2_{13}\DuCorep{V}^1_{12} \in
  \U(\Hils\otimes\Hils[K]^1\otimes\Hils[K]^2).
  \]
  Left representations of~\(\Multunit\)
  also form a Hilbert space tensor category with the obvious forgetful
  functor and \(\Trivial(\Hils)=(\Hils,1)\).
  Actually, this tensor category is \emph{isomorphic} to the category
  of right representations of the dual multiplicative unitary
  \(\DuMultunit= \Sigma \Multunit[*]\Sigma\):
  the isomorphism takes a left representation
  \(\DuCorep{V}\in \U(\Hils[K]\otimes\Hils)\)
  to the right representation
  \(\Sigma \DuCorep{V}^* \Sigma \in \U(\Hils\otimes\Hils[K])\)
  of~\(\DuMultunit\).
  Since~\(\DuMultunit\)
  is a natural right absorber for right representations
  of~\(\DuMultunit\)
  by Example~\ref{exa:Multunit_absorber}, the unitary~\(\Multunit\),
  viewed as a left representation, is a natural \emph{right}
  absorber in the tensor category of left representations
  of~\(\Multunit\).  The natural intertwiner is
  \[
  \Sigma \DuCorep{V}^*\Sigma\colon
  (\Hils[K]\otimes\Hils, \DuCorep{V}\tenscorep\Multunit)
  \to (\Hils[K]\otimes\Hils, 1_{\Hils[K]}\tenscorep\Multunit).
  \]
\end{example}

Next we want to prove that the multiplicative unitaries for two
natural right absorbers of~\(\Cat\) are isomorphic in the category of
multiplicative unitaries introduced in~\cite{Ng:Morph_of_Mult_unit}
and further studied in~\cite{Meyer-Roy-Woronowicz:Homomorphisms}.

\begin{proposition}
  \label{pro:compare_absorbers}
  Let \((\rho,(U^x)_{x\inOb\Cat})\)
  and \((\check{\rho},(\check{U}{}^x)_{x\inOb\Cat})\)
  be two natural right absorbers for \((\Cat,\Forget)\).
  Let \(\Hils\defeq \Hils^\rho\),
  \(\check\Hils\defeq \Hils^{\check\rho}\),
  \(U\defeq U^\rho \in \U(\Hils\otimes\Hils)\),
  \(\check{U}\defeq U^{\check\rho} \in
  \U(\check{\Hils}\otimes\check{\Hils})\)
  be the corresponding multiplicative unitaries.  The unitaries
  \[
  V \defeq U^{\check\rho} \in \U(\check{\Hils} \otimes \Hils),\qquad
  W \defeq \check{U}^\rho \in \U(\Hils \otimes \check{\Hils})
  \]
  satisfy the following pentagon-like equations:
  \begin{alignat*}{2}
    U_{23} V_{12} &= V_{12} V_{13} U_{23},&\qquad
    \check{U}_{23} W_{12} &= W_{12} W_{13} \check{U}_{23},\\
    V_{23} \check{U}_{12} &= \check{U}_{12} V_{13} V_{23},&\qquad
    W_{23} U_{12} &= U_{12} W_{13} W_{23},\\
    V_{23} W_{12} &= W_{12} U_{13} V_{23},&\qquad
    W_{23} V_{12} &= V_{12} \check{U}_{13} W_{23}.
  \end{alignat*}
  If the multiplicative unitaries \(U\)
  and~\(\check{U}\)
  are manageable, then \(V\)
  and~\(W\)
  give morphisms between the corresponding \(\Cst\)\nb-quantum
  groups that are inverse to each other in the category of
  \(\Cst\)\nb-quantum
  groups defined
  in~\textup{\cite{Meyer-Roy-Woronowicz:Homomorphisms}}.
\end{proposition}

\begin{proof}
  Our assumptions are symmetric in \((\rho,U)\)
  and \((\check{\rho},\check{U})\).
  When we exchange them, the equations in the first column become the
  corresponding ones in the second column.  So it suffices to prove
  those in the first column.  We already know that
  \(V=U^{\check\rho}\)
  is a right representation of~\(U\),
  which gives the first equation.  The other equations are proved
  similarly.  For the second equation, we use the naturality of~\(U\)
  for the intertwiner
  \(\check{U}\colon \check\rho\otimes \check\rho \to
  \Trivial(\check\rho) \otimes \check\rho\)
  and rewrite
  \(U^{\check\rho\otimes \check\rho} = U^{\check\rho}_{13}
  U^{\check\rho}_{23} = V_{13} V_{23}\)
  and
  \(U^{\Trivial(\check\rho)\otimes \check\rho} = U^{\check\rho}_{23} =
  V_{23}\).
  For the third equation, we use the naturality of~\(U\)
  for the intertwiner
  \(W\colon \rho\otimes \check\rho \to \Trivial(\rho) \otimes
  \check\rho\)
  and rewrite
  \(U^{\rho\otimes \check\rho} = U^{\rho}_{13} U^{\check\rho}_{23} =
  U_{13} V_{23}\)
  and
  \(U^{\Trivial(\check\rho)\otimes \check\rho} = U^{\check\rho}_{23} =
  V_{23}\).

  Morphisms of quantum groups are described in
  \cite{Meyer-Roy-Woronowicz:Homomorphisms}*{Lemma 3.2}.  The
  equations \(U_{23} V_{12} = V_{12} V_{13} U_{23}\)
  and \(V_{23} \check{U}_{12} = \check{U}_{12} V_{13} V_{23}\)
  say that~\(V\)
  is a morphism from~\(\check{U}\)
  to~\(U\).
  The equations
  \(\check{U}_{23} W_{12} = W_{12} W_{13} \check{U}_{23}\)
  and \(W_{23} U_{12} = U_{12} W_{13} W_{23}\)
  say that~\(W\)
  is a morphism from~\(U\)
  to~\(\check{U}\).
  The product of two morphisms is defined
  in~\cite{Meyer-Roy-Woronowicz:Homomorphisms}*{Definition 3.5} as the
  solution of a certain operator equation.  The equation
  \(V_{23} W_{12} = W_{12} U_{13} V_{23}\)
  says that the product of \(V\)
  and~\(W\)
  is~\(U\).
  The equation \(W_{23} V_{12} = V_{12} \check{U}_{13} W_{23}\)
  says that the product of \(W\)
  and~\(V\)
  is~\(\check{U}\).
  Manageability is needed in~\cite{Meyer-Roy-Woronowicz:Homomorphisms}
  to ensure that the equation
  in~\cite{Meyer-Roy-Woronowicz:Homomorphisms}*{Definition 3.5} always
  has a solution.  So manageability is needed to talk about a
  \emph{category} of morphisms between multiplicative unitaries.
\end{proof}

\begin{example}
  Let \((\rho,U)\) be a natural right absorber for~\(\Cat\) and let
  \(y\inOb\Cat\).  Then \(\check\rho = \rho\otimes y\) with
  \(\check{U}{}^x \defeq U^x \otimes \Id_y\) for all \(x\inOb\Cat\)
  is a natural right absorber as well.  The corresponding
  multiplicative unitary is
  \begin{equation}
    \label{eq:stabilise_multunit}
    \check{U}{}^{\rho\otimes y}
    = (U^{\rho\otimes y})^{\phantom1}_{123}
    = U^\rho_{13} U^y_{23}
    \in \U(\Hils^\rho\otimes\Hils^y \otimes \Hils^\rho\otimes\Hils^y).
  \end{equation}
  Proposition~\ref{pro:compare_absorbers} shows that \(U^\rho\)
  and~\(\check{U}{}^{\rho\otimes y}\)
  are isomorphic multiplicative unitaries when they are both
  manageable, compare \cite{Meyer-Roy-Woronowicz:Qgrp_proj}*{Theorem
    3.7}.
\end{example}

We now extend the analysis above to describe functors between
representation categories.
Let \(\Cat_1\)
and~\(\Cat_2\)
be Hilbert space tensor categories with natural right absorbers
\((\rho_1,U_1)\)
and \((\rho_2,U_2)\),
respectively.  Let \(\Phi\colon \Cat_1\to\Cat_2\)
be a strict tensor functor with \(\Forget_2\circ \Phi=\Forget_1\).
If \(\Cat_1\)
and~\(\Cat_2\)
are representation categories, then this means that~\(\Phi\)
turns a representation of one sort into one of the other \emph{on
  the same Hilbert space} in a natural way and preserving tensor
products.
Such a functor also satisfies \(\Phi\circ\Trivial_1 = \Trivial_2\)
by the argument in Remark~\ref{rem:trivial_unique}.

\begin{proposition}
  \label{pro:functor_Corep_morphism}
  The unitary
  \(V^\Phi \defeq U_2^{\Phi(\rho_1)}\in
  \U(\Hils^{\rho_1}\otimes\Hils^{\rho_2})\) satisfies
  \[
  (U_2^{\rho_2})^{\phantom1}_{23} V^\Phi_{12}
  = V^\Phi_{12} V^\Phi_{13} (U_2^{\rho_2})^{\phantom1}_{23},\qquad
  V^\Phi_{23} (U_1^{\rho_1})^{\phantom1}_{12}
  = (U_1^{\rho_1})^{\phantom1}_{12} V^\Phi_{13} V^\Phi_{23},
  \]
  that is, \(V^\Phi\) is a bicharacter from~\(U_1^{\rho_1}\)
  to~\(U_2^{\rho_2}\).  Moreover, for
  any \(x\inOb\Cat_1\),
  \begin{equation}
    \label{eq:functor_Corep_morphism}
    V^\Phi_{23} (U_1^x)^{\phantom1}_{12}
    = (U_1^x)^{\phantom1}_{12} (U_2^{\Phi(x)})^{\phantom1}_{13} V^\Phi_{23}
    \in \U(\Hils^x\otimes\Hils^{\rho_1}\otimes\Hils^{\rho_2}).
  \end{equation}
  If the multiplicative unitary~\(U_1^{\rho_1}\)
  is manageable and \(\Cat_2 \cong \Corepcat{U_2^{\rho_2}}\),
  then the map from functors \(\Phi\colon \Cat_1\to \Cat_2\)
  as above to unitary bicharacters
  \(V\in \U(\Hils^{\rho_1}\otimes\Hils^{\rho_2})\)
  from~\(U_1^{\rho_1}\)
  to~\(U_2^{\rho_2}\)
  is bijective.  If the multiplicative unitaries \(U_1^{\rho_1}\)
  and~\(U_2^{\rho_2}\)
  are both manageable, then \(V^\Phi\)
  is a morphism between the corresponding \(\Cst\)\nb-quantum
  groups in the category defined
  in~\textup{\cite{Meyer-Roy-Woronowicz:Homomorphisms}}.
\end{proposition}

\begin{proof}
  The first two equations in the proposition say that~\(V^\Phi\)
  is a morphism of \(\Cst\)\nb-quantum
  groups as in \cite{Meyer-Roy-Woronowicz:Homomorphisms}*{Lemma 3.2}
  provided the multiplicative unitaries \(U_1^{\rho_1}\)
  and~\(U_2^{\rho_2}\)
  are manageable, so that they generate \(\Cst\)\nb-quantum
  groups.  We already know that~\(V^\Phi\)
  is a right representation of~\(U_2^{\rho_2}\),
  which is the first equation.  The second equation is the special
  case \(x=\rho_1\)
  of~\eqref{eq:functor_Corep_morphism}.
  Equation~\ref{eq:functor_Corep_morphism} says that the functor on
  representation categories induced by~\(V^\Phi\)
  is~\(\Phi\),
  as expected.  To prove~\eqref{eq:functor_Corep_morphism}, we
  identify
  \begin{alignat*}{2}
    \Phi(x \otimes\rho_1) &= \Phi(x) \otimes\Phi(\rho_1),&\quad
    U_2^{\Phi(x \otimes\rho_1)} &=
%   (U_2^{\Phi(x)})^{\phantom1}_{13} (U_2^{\Phi(\rho_1)})^{\phantom1}_{23} =
    (U_2^{\Phi(x)})^{\phantom1}_{13} V^{\Phi}_{23},\\
    \Phi(\Trivial(x) \otimes\rho_1) &= \Trivial(\Phi(x))
    \otimes\Phi(\rho_1),&\quad
    U_2^{\Phi(\Trivial(x) \otimes\rho_1)} &=
    (U_2^{\Phi(\rho_1)})^{\phantom1}_{23} = V^{\Phi}_{23}.
  \end{alignat*}
  The naturality of~\(U_2\)
  for the intertwiner
  \(\Phi(U_1^x)\colon \Phi(x \otimes\rho_1) \to \Phi(\Trivial(x)
  \otimes\rho_1)\)
  gives~\eqref{eq:functor_Corep_morphism}.  This is equivalent to
  \((U_2^{\Phi(x)})^{\phantom1}_{13} = (U_1^x)_{12}^{*\phantom1}
  V^\Phi_{23} (U_1^x)^{\phantom1}_{12} (V^\Phi_{23})^*\),
  which determines the object~\(\Phi(x)\)
  of~\(\Cat_2\)
  by Proposition~\ref{pro:absorber}.  This describes how~\(\Phi\)
  acts on objects.  Then its action on arrows is determined by the
  faithful forgetful functor to Hilbert spaces.  So~\(V^\Phi\)
  determines the functor~\(\Phi\).

  Now assume that~\(U_1^{\rho_1}\)
  is manageable.  Let \(V\in \U(\Hils^{\rho_1}\otimes\Hils^{\rho_2})\)
  be a bicharacter.  Any bicharacter induces a functor~\(\Phi\)
  between the representation categories by
  \cite{Meyer-Roy-Woronowicz:Homomorphisms}*{Proposition~6.5}.  The
  proof of this proposition does not describe this functor~\(\Phi\)
  explicitly.  An explicit formula for~\(\Phi\)
  is similar to the formula for the composition of bicharacters, which
  is a special case.  Namely, let \(x\inOb\Cat_1\).
  As in the proof of \cite{Meyer-Roy-Woronowicz:Homomorphisms}*{Lemma
    3.6}, manageability shows that there is a unitary
  operator~\(U_2^{\Phi(x)}\)
  that verifies~\eqref{eq:functor_Corep_morphism}; moreover,
  \(U_2^{\Phi(x)}\)
  is a representation of~\(U_2^{\rho_2}\),
  and there is a unique functor
  \(\Phi\colon \Cat_1 \to \Corepcat{U_2^{\rho_2}}\)
  with \(\Forget\circ\Phi=\Forget\)
  that sends \(x\inOb\Cat\)
  to this representation and that acts by the identity map on arrows,
  viewed as Hilbert space operators.  This functor is a strict tensor
  functor.  Any functor~\(\Phi\)
  as above is of this form for the corresponding bicharacter~\(V^\Phi\).
  This gives the desired bijection.
\end{proof}

Proposition~\ref{pro:functor_Corep_morphism} gives yet another
equivalent characterisation of the quantum group morphisms
of~\cite{Meyer-Roy-Woronowicz:Homomorphisms}: they are equivalent to
strict tensor functors between the representation categories with
\(\Forget\circ \Phi = \Forget\).
This result is similar in spirit to
\cite{Meyer-Roy-Woronowicz:Homomorphisms}*{Theorem 6.1}, which uses
coactions on \(\Cst\)\nb-algebras instead of representations.

\subsection{Left and right absorbers}
\label{sec:left_absorbers}

A natural left absorber in~\(\Cat\) is defined like a natural right
absorber, but on the other side:

\begin{definition}
  \label{def:natural_left_absorber}
  A \emph{natural left absorber} in~\(\Cat\) is an object
  \(\lambda\inOb\Cat\) with unitaries
  \[
  U_\lambda^x\colon
  \lambda \otimes x \to \lambda \otimes \Trivial(x)\qquad
  \text{for all }x\inOb\Cat
  \]
  with the following properties:
  \begin{enumerate}[label=\textup{(\ref*{def:natural_left_absorber}.\arabic*)}]
  \item the unitaries~\(U_\lambda^x\)
    are natural, that is, the following diagram commutes for
    any arrow \(a\colon x_1\to x_2\):
    \[
    \begin{tikzpicture}
      \matrix (m) [cd] {
        \lambda \otimes x_1 &
        \lambda \otimes \Trivial(x_1)\\
        \lambda \otimes x_2&
        \lambda \otimes \Trivial(x_2)\\
      };
      \labelar{m-1-1}{U_\lambda^{x_1}}{\cong}{m-1-2};
      \labelar{m-2-1}{U_\lambda^{x_2}}{\cong}{m-2-2};
      \labelar{m-1-1}{}{\Id_\lambda \otimes a}{m-2-1};
      \labelar{m-1-2}{\Id_\lambda \otimes a}{}{m-2-2};
    \end{tikzpicture}
    \]
  \item for all \(x_1,x_2\inOb\Cat\), the following diagram
    commutes:
    \[
    \begin{tikzpicture}
      \matrix (m) [cd,column sep=3em] {
        \lambda \otimes x_1\otimes x_2 &
        \lambda \otimes \Trivial(x_1\otimes x_2)\\
        \lambda \otimes \Trivial(x_1)\otimes x_2&
        \lambda \otimes \Trivial(x_1)\otimes \Trivial(x_2)\\
      };
      \labelar{m-1-1}{U_\lambda^{x_1\otimes x_2}}{}{m-1-2};
      \labelar{m-1-1}{}{(U_\lambda^{x_1})_{12}}{m-2-1};
      \labelar{m-2-1}{(U_\lambda^{x_2})_{13}}{}{m-2-2};
      \draw[double,double equal sign distance] (m-1-2) -- (m-2-2);
    \end{tikzpicture}
    \]
  \end{enumerate}
\end{definition}

The analogue of Lemma~\ref{lem:natural_absorber_Trivial} holds for
natural left absorbers as well, that is,
\(U_\lambda^{\Trivial(\Hils)}= \Id_{\Trivial(\Hils) \otimes\lambda}\)
for any Hilbert space~\(\Hils\).

Let~\(\Multunit\) be a multiplicative unitary.  Then the categories
of left and of right representations of~\(\Multunit\) have a
canonical natural \emph{right} absorber by Examples
\ref{exa:Multunit_absorber} and~\ref{exa:left_corep}.  It is
unclear, in general, whether they have a natural left absorber as
well.  The only construction of left absorbers that we know uses the
contragradient operation to turn a right into a left absorber.  For
contragradients to exist, we assume~\(\Multunit\) to be manageable.
We work with right representations of~\(\Multunit\).  The
contragradient of a representation~\(\Corep{U}\) on
a Hilbert space~\(\Hils\) is a
representation~\(\widetilde{\Corep{U}}\) on the complex-conjugate
Hilbert space~\(\conj{\Hils}\).  The contragradient construction
becomes a covariant functor \(\Corepcat{\Multunit} \to
\Corepcat{\Multunit}\) when we map an intertwiner \(a\colon
\Hils_1\to\Hils_2\) to \(\conj{a}\colon \conj{\Hils_1} \to
\conj{\Hils_2}\).  This is not quite a \(\Wst\)\nb-functor because it
is conjugate-linear, not linear.  The
contragradient of a trivial representation remains trivial.  The
contragradient operation is involutive, that is,
\(\widetilde{\widetilde{\Corep{U}}} = \Corep{U}\) for
representations and \(\conj{\conj{a}} = a\) for
intertwiners.  It reverses the order of tensor factors: the flip
operator \(\Sigma\colon \conj{\Hils_1} \otimes \conj{\Hils_2} \to
\conj{\Hils_2} \otimes \conj{\Hils_1} \cong
\conj{\Hils_2\otimes\Hils_1}\) intertwines \(\widetilde{\Corep{U}_1}
\tenscorep \widetilde{\Corep{U}_2}\) with the contragradient of
\(\Corep{U}_2 \tenscorep \Corep{U}_1\), see
\cite{Soltan-Woronowicz:Multiplicative_unitaries}*{Section~3}.

Let \((\rho,(U^x)_{x\inOb\Corepcat{\Multunit}})\) be a natural right
absorber for~\(\Corepcat{\Multunit}\).  For instance, we may take
the canonical one described in Example~\ref{exa:Multunit_absorber}.
Let \(\lambda \defeq \widetilde\rho\) be the contragradient
of~\(\rho\), so \(\widetilde\lambda=\rho\).  Let \(U_\lambda^x\colon
\lambda \otimes x \to \Trivial(\lambda)\otimes x\) for
\(x\inOb\Corepcat{\Multunit}\) be the composite unitary intertwiner
\[
\lambda \otimes x =
\widetilde{\widetilde{\lambda \otimes x}} \xrightarrow{\Sigma}
\widetilde{\tilde{x} \otimes \tilde{\lambda}}
\xrightarrow{\conj{U^{\tilde{x}}}}
\widetilde{\Trivial(\tilde{x}) \otimes \rho}
=
\widetilde{\widetilde{\Trivial(x}) \otimes \widetilde{\lambda}}
 \xrightarrow{\Sigma}
\widetilde{\widetilde{\lambda \otimes \Trivial(x)}}
=
\lambda \otimes \Trivial(x).
\]
Routine computations show that \((\lambda,(U_\lambda^x))\) is a
natural left absorber if \((\rho,(U^x))\) is a natural right
absorber.  This proves the following:

\begin{proposition}
  \label{pro:left_right_absorber}
  Let~\(\Multunit\) be a multiplicative unitary.  If~\(\Multunit\)
  is manageable, then its tensor category of
  representations~\(\Corepcat{\Multunit}\) contains both a natural
  right absorber and a natural left absorber.
\end{proposition}

If~\(\Cat\) has both a right absorber~\(\rho\) and a left
absorber~\(\lambda\), then
\[
\Trivial(\lambda) \otimes \rho \cong \lambda\otimes\rho
\cong \lambda\otimes\Trivial(\rho).
\]
Hence the direct sums of infinitely many copies of \(\lambda\)
and~\(\rho\) are isomorphic.  This common direct sum is both a left
and a right absorber, and its isomorphism class does not depend on
the choice of \(\lambda\) or~\(\rho\).  These observations go back
already to~\cite{Pinzari-Roberts:Regular}, and they need only
absorption, without any naturality.  We are going to use the
uniqueness of two-sided absorbers to prove that any isomorphism
between the representation categories of two \(\Cst\)\nb-quantum
groups comes from an isomorphism of Hopf \Star{}algebras.  First we
need a preparatory result, which would really belong
into~\cite{Meyer-Roy-Woronowicz:Homomorphisms}, but was not proved
there.

\begin{theorem}
  \label{the:isomorphisms_quantum_groups}
  The isomorphisms in the category of \(\Cst\)\nb-quantum
  groups defined in~\cite{Meyer-Roy-Woronowicz:Homomorphisms} are the
  same as the Hopf \Star{}isomorphisms of the underlying
  \(\Cst\)\nb-bialgebras.
\end{theorem}

\begin{proof}
  It is trivial that a Hopf \Star{}isomorphism induces an isomorphism
  in the category of~\cite{Meyer-Roy-Woronowicz:Homomorphisms}.
  Conversely, an isomorphism between two \(\Cst\)\nb-quantum
  groups \((C_i,\Comult[C_i])\),
  \(i=1,2\),
  in the category of~\cite{Meyer-Roy-Woronowicz:Homomorphisms} only
  gives a Hopf \Star{}isomorphism between their universal dual quantum
  groups \(\hat{C}_1^\univ \cong \hat{C}_2^\univ\)
  (or \(C_1^\univ \cong C_2^\univ\),
  but we shall use the dual isomorphism below).  For locally compact
  quantum groups with Haar weights, an isomorphism
  \(C_1^\univ \cong C_2^\univ\)
  implies a Hopf \Star{}isomorphism between \((C_1,\Comult[C_1])\)
  and \((C_2,\Comult[C_2])\)
  because the invariant weights on \(C_1^\univ \cong C_2^\univ\)
  are unique, see \cite{Kustermans-Vaes:LCQG}*{p.~873}.  We shall
  generalise this to \(\Cst\)\nb-quantum
  groups generated by manageable multiplicative unitaries.
  The Hopf \Star{}isomorphism \(\hat{C}_1^\univ \cong
  \hat{C}_2^\univ\) induces an isomorphism between the
  representation categories of \((C_1,\Comult[C_1])\) and
  \((C_2,\Comult[C_2])\).

  Let \(\Multunit_i\in\U(\Hils_i\otimes\Hils_i)\), \(i=1,2\), be
  manageable multiplicative unitaries that generate
  \((C_i,\Comult[C_i])\).  We view~\(\Multunit_i\) as a right
  representation of \((C_i,\Comult[C_i])\) on~\(\Hils_i\).  The
  representation of~\(\hat{C}_i^\univ\) associated
  to~\(\Multunit_i\) descends to a faithful representation
  of~\(\hat{C}_i\): this is the standard construction of
  \(\hat{C}_i\subseteq \Bound(\Hils_i)\) from a multiplicative
  unitary in~\cite{Soltan-Woronowicz:Multiplicative_unitaries}.
  Thus we have to prove that the representations of
  \(\hat{C}_1^\univ \cong \hat{C}_2^\univ\) associated to
  \(\Multunit_1\) and~\(\Multunit_2\) have the same kernel.  Since
  our multiplicative unitaries are manageable, the representation
  category
  \[
  \Cat \defeq \Corepcat{\Multunit_1}
  \cong \Corepcat{(C_1,\Comult[C_1])}
  \cong \Corepcat{(C_2,\Comult[C_2])}
  \cong \Corepcat{\Multunit_2}
  \]
  contains both a natural left and a natural right absorber by
  Proposition~\ref{pro:left_right_absorber}.  Both \(\Multunit_1\)
  and~\(\Multunit_2\) are natural right absorbers.  By the remarks
  above, the direct sums \((\Multunit_1)^\infty\)
  and~\((\Multunit_2)^\infty\) of infinitely many copies of
  \(\Multunit_1\) and~\(\Multunit_2\) are isomorphic objects
  of~\(\Cat\) because they are both isomorphic to the direct sum of
  infinitely many copies of a left absorber.  Therefore, the
  representations of~\(\hat{C}_1^\univ\) associated to
  \((\Multunit_1)^\infty\) and~\((\Multunit_2)^\infty\) have the same
  kernel.  Then the representations of~\(\hat{C}_1^\univ\)
  associated to \(\Multunit_1\) and~\(\Multunit_2\) also have the
  same kernel.  Thus our Hopf \Star{}isomorphism
  \(\hat{C}_1^\univ\cong \hat{C}_2^\univ\) descends to a Hopf
  \Star{}isomorphism \(\hat{C}_1\cong \hat{C}_2\).  This implies a
  Hopf \Star{}isomorphism \(C_1\cong C_2\).
\end{proof}

\begin{corollary}
  \label{cor:Tannaka-Krein}
  A \(\Cst\)\nb-quantum
  group \((C,\Comult[C])\)
  is determined uniquely by its tensor category
  \(\Corepcat{C,\Comult[C]}\)
  of representations with the forgetful functor to~\(\Hilb\).
\end{corollary}

\begin{proof}
  Assume to begin with that there is an equivalence of tensor
  categories~\(F_0\)
  from \(\Corepcat{C,\Comult[C]}\)
  to \(\Corepcat{D,\Comult[D]}\)
  such that the forgetful functors \(\Forget\circ F_0\)
  and~\(\Forget\)
  to~\(\Hilb\)
  are naturally isomorphic.  This natural isomorphism consists of
  natural unitaries
  \(\Upsilon_{(\Hils,V)}\colon \Forget(F_0(\Hils,V)) \congto \Hils\)
  for all Hilbert spaces~\(\Hils\)
  with a representation~\(V\)
  of~\((C,\Comult[C])\).
  We use~\(\Upsilon_{(\Hils,V)}\)
  on the first leg to transfer the representation~\(F_0(\Hils,V)\)
  of~\((D,\Comult[D])\)
  to the Hilbert space~\(\Hils\).
  This gives another equivalence of tensor categories~\(F\)
  from \(\Corepcat{C,\Comult[C]}\)
  to \(\Corepcat{D,\Comult[D]}\)
  such that the tensor functors \(\Forget\circ F\)
  and~\(\Forget\)
  are equal.  Thus~\(F\)
  turns a representation of~\((C,\Comult[C])\)
  on a Hilbert space~\(\Hils\)
  into a representation of~\((D,\Comult[D])\)
  on the same Hilbert space and maps an intertwiner
  for~\((C,\Comult[C])\)
  to the same operator, now as an intertwiner for~\((D,\Comult[D])\).
  Since the forgetful functor to Hilbert spaces is faithful and
  strict, the functor~\(F\)
  is a strict tensor functor as well.  We may improve the inverse
  equivalence to a strict tensor functor acting identically on objects
  as well.  Thus~\(F\) is an isomorphism of tensor categories.

  Let \(\Multunit[C]\)
  and~\(\Multunit[D]\)
  be manageable multiplicative unitaries that generate
  \((C,\Comult[C])\)
  and \((D,\Comult[D])\).
  A representation of \((C,\Comult[C])\)
  is equivalent to one of~\(\Multunit[C]\)
  on the same Hilbert space.  So
  \(\Corepcat{C,\Comult[C]} = \Corepcat{\Multunit[C]}\).
  Similarly, \(\Corepcat{D,\Comult[D]} = \Corepcat{\Multunit[D]}\).
  So \(\Multunit[C]\)
  and~\(\Multunit[D]\)
  are natural right absorbers in
  \(\Corepcat{C,\Comult[C]} \cong \Corepcat{D,\Comult[D]}\)
  by Example~\ref{exa:Multunit_absorber}.  By
  Proposition~\ref{pro:compare_absorbers}, the multiplicative
  unitaries \(\Multunit[C]\)
  and~\(\Multunit[D]\)
  are isomorphic in the category
  of~\cite{Meyer-Roy-Woronowicz:Homomorphisms}.
  Theorem~\ref{the:isomorphisms_quantum_groups} shows that this
  isomorphism gives a Hopf \Star{}isomorphism
  \((C,\Comult[C]) \cong (D,\Comult[D])\).
\end{proof}

Proposition~\ref{pro:absorber} has a variant for natural left
absorbers.  Let \(\lambda\)
and~\((U^x)_{x\inOb\Cat}\)
be a natural left absorber for~\(\Cat\).
For \(x\inOb\Cat\),
let \(\Hils^x \defeq \Forget(x)\),
and write~\(U^x\)
for \(\Forget(U^x)\in \U(\Hils^\lambda\otimes\Hils^x)\).
Then~\(U^\lambda\) is an ``antimultiplicative'' unitary:
\[
U^\lambda_{12}U^\lambda_{23}
= U^\lambda_{23}U^\lambda_{13}U^\lambda_{12}.
\]
Moreover, \(U^x\)
for \(x\inOb\Cat\) is a left representation of~\(U^\lambda\):
\[
U^{x}_{23}U^{x}_{13}U^\lambda_{12}
= U^\lambda_{12}U^x_{23}.
\]
We define a tensor product for representations of~\(U^\lambda\) by
\[
U \tenscorep V \defeq V_{13}U_{12}.
\]
The map \(x\mapsto U^x\)
gives a fully faithful, strict tensor functor from~\(\Cat\)
to~\(\Corepcat{U^\lambda}\),
which intertwines the forgetful functors from \(\Cat\)
and \(\Corepcat{U^\lambda}\) to~\(\Hilb\).

Similarly, there is an analogue of
Proposition~\ref{pro:compare_absorbers}, saying that the
antimultiplicative unitaries \(\Hils\defeq \Hils^\lambda\),
\(\check\Hils\defeq \Hils^{\check\lambda}\),
\(U\defeq U^\lambda\),
\(\check{U}\defeq \check{U}^{\check\lambda}\)
associated to two natural left absorbers
\((\lambda, (U^x)_{x\inOb\Cat})\)
and \((\check{\lambda},(\check{U}^x)_{x\inOb\Cat})\)
are ``isomorphic'' in a suitable sense.  Namely, the unitaries
\[
V \defeq U^{\check\lambda} \in \U(\Hils\otimes\check{\Hils}),\qquad
W \defeq \check{U}^\lambda \in \U(\check{\Hils}\otimes\Hils)
\]
satisfy the following pentagon-like equations:
\begin{alignat*}{2}
  U_{12} V_{23} &= V_{23} V_{13} U_{12},&\qquad
  \check{U}_{12} W_{23} &= W_{23} W_{13} \check{U}_{12},\\
  V_{12} \check{U}_{23} &= \check{U}_{23} V_{13} V_{12},&\qquad
  W_{12} U_{23} &= U_{23} W_{13} W_{12},\\
  V_{12} W_{23} &= W_{23} U_{13} V_{12},&\qquad
  W_{12} V_{23} &= V_{23} \check{U}_{13} W_{12}.
\end{alignat*}
It is also interesting to apply the same technique to a tensor
category with a natural right absorber \((\rho,(U^x)_{x\inOb\Cat})\)
and a natural left absorber
\((\lambda,(\check{U}^x)_{x\inOb\Cat})\).
Let \(\Hils\defeq \Hils^\rho\),
\(\check\Hils\defeq \Hils^{\lambda}\),
\(U\defeq U^\rho\),
\(\check{U}\defeq \check{U}^{\lambda}\)
be the associated multiplicative and antimultiplicative unitaries.  Define
\[
V \defeq U^{\lambda} \in \U(\check{\Hils}\otimes\Hils),\qquad
W \defeq \check{U}^\rho \in \U(\check{\Hils}\otimes\Hils).
\]
These unitaries satisfy the following pentagon-like equations:
\begin{alignat*}{2}
  \check{U}_{12} V_{23} &= V_{23} V_{13} \check{U}_{12},&\qquad
  \check{U}_{12} W_{23} &= W_{23} W_{13} \check{U}_{12},\\
  U_{23} V_{12} &=  V_{12} V_{13} U_{23},&\qquad
  U_{23} W_{12}  &= W_{12} W_{13} U_{23},\\
  V_{13} W_{12} &= W_{12} V_{13} U_{23},&\qquad
  W_{13} V_{23} &= V_{23} W_{13}\check{U}_{13}.
\end{alignat*}
The proofs are similar to those in
Proposition~\ref{pro:compare_absorbers}.  In addition, let~\(x\)
be any object of~\(\Cat\).
Naturality of~\(\check{U}\)
with respect to the intertwiner
\(U^x\colon x\otimes \rho \to \Trivial(x)\otimes\rho\) gives
\begin{equation}
  \label{eq:rep_to_anti-rep}
  \check{U}^x_{12}
  = (\check{U}^\rho_{13})^* (U^x_{23})^* \check{U}^\rho_{13} U^x_{23}
  = W_{13}^* (U^x_{23})^* W_{13}^{\phantom{*}} U^x_{23}.
\end{equation}
Naturality of~\(U\)
with respect to the intertwiner
\(\check{U}^x\colon \lambda\otimes x \to \lambda\otimes\Trivial(x)\)
gives
\begin{equation}
  \label{eq:anti-rep_to_rep}
  U^x_{23}
  = (U^\lambda_{23})^* (\check{U}^x_{12})^* U^\lambda_{23} \check{U}^x_{12}
  = V_{23}^* (\check{U}^x_{12})^* V_{23}^{\phantom{*}} \check{U}^x_{12}.
\end{equation}
Here \(U^x\)
and~\(\check{U}^x\)
are the representations of \(U\)
and~\(\check{U}\)
associated to~\(x\),
respectively.  So these determine each other.  If
\(\Cat=\Corepcat{U}\)
for a manageable multiplicative unitary~\(U\)
and~\(\check{U}\)
comes from its contragradient as above, then also
\(\Cat=\Corepcat{\check{U}}\).
So for a given representation~\(U^x\)
of~\(U\),
there is a unique representation~\(\check{U}^x\)
of~\(\check{U}\)
satisfying~\eqref{eq:rep_to_anti-rep}.  And for a given
representation~\(\check{U}^x\)
of~\(\check{U}\),
there is a unique representation~\(U^x\)
of~\(U\) satisfying~\eqref{eq:anti-rep_to_rep}.

Multiplicative and antimultiplicative unitaries are closely related to
the Heisenberg and anti-Heisenberg pairs studied
in~\cite{Meyer-Roy-Woronowicz:Twisted_tensor}.  By definition, a
Heisenberg pair for a \(\Cst\)\nb-quantum
group \((C,\Comult[C])\)
is a pair of representations \((\pi,\hat\pi)\)
of \((C,\hat{C})\)
such that \((\hat{\pi}\otimes\pi)\multunit\)
for the reduced bicharacter \(\multunit\in\U(\hat{C}\otimes C)\)
is a multiplicative unitary.  And an anti-Heisenberg pair is a pair of
representations \((\sigma,\hat\sigma)\)
of \((C,\hat{C})\)
such that \((\hat{\sigma}\otimes\sigma)\multunit\)
is an antimultiplicative unitary.

\section{Representations of braided multiplicative unitaries}
\label{sec:corep_braided}

Let \(\Hils\)
and~\(\Hils[L]\)
be Hilbert spaces and let \(\Multunit \in \U(\Hils\otimes\Hils)\)
be a multiplicative unitary.  Let
\[
\Corep{U}\in\U(\Hils[L]\otimes\Hils),\qquad
\DuCorep{V}\in\U(\Hils\otimes\Hils[L]),\qquad
\BrMultunit\in\U(\Hils[L]\otimes\Hils[L]),
\]
be a braided multiplicative unitary over~\(\Multunit\)
(see~\cite{Meyer-Roy-Woronowicz:Qgrp_proj}).  We are first going to
define a tensor category
\(\Corepcat{\Multunit,\Corep{U},\DuCorep{V},\BrMultunit}\)
of right representations.

\begin{definition}
  \label{def:braided_corep}
  A (right) \emph{representation} of
  \((\Multunit,\Corep{U},\DuCorep{V},\BrMultunit)\)
  is a triple \((\Hils[K],\Corep{S},\Corep{T})\),
  where~\(\Hils[K]\)
  is a Hilbert space, \(\Corep{S}\in \U(\Hils[K]\otimes\Hils)\)
  is a right representation of~\(\Multunit\)
  on~\(\Hils[K]\), that is,
  \begin{equation}
    \label{eq:S_corep}
    \Multunit_{23} \Corep{S}_{12}
    = \Corep{S}_{12} \Corep{S}_{13} \Multunit_{23}
    \quad\text{in }\U(\Hils[K]\otimes\Hils\otimes\Hils),
  \end{equation}
  and \(\Corep{T}\in\U(\Hils[K]\otimes\Hils[L])\)
  is equivariant with respect to the tensor product representation
  \(\Corep{S}\tenscorep\Corep{U}\) of~\(\Multunit\),
  \begin{equation}
    \label{eq:T_SU-invariant}
    \Corep{S}_{13} \Corep{U}_{23} \Corep{T}_{12}
    = \Corep{T}_{12} \Corep{S}_{13} \Corep{U}_{23}
    \quad\text{in }\U(\Hils[K]\otimes\Hils[L]\otimes \Hils),
  \end{equation}
  and satisfies the (top-braided) representation condition
  \begin{equation}
    \label{eq:T_corep}
    \BrMultunit_{23} \Corep{T}_{12}
    = \Corep{T}_{12} (\Braiding{\Hils[L]}{\Hils[L]})^{\phantom1}_{23}
    \Corep{T}_{12} (\Dualbraiding{\Hils[L]}{\Hils[L]})^{\phantom1}_{23}
    \BrMultunit_{23}
    \quad\text{in }\U(\Hils[K]\otimes\Hils[L]\otimes\Hils[L]).
  \end{equation}
\end{definition}

We recall how the braiding operators
\(\Braiding{\Hils[L]}{\Hils[K]}\colon \Hils[L]\otimes\Hils[K]
\to\Hils[K]\otimes\Hils[L]\)
are defined, where~\(\Hils[K]\)
carries a representation \(\Corep{S}\in \U(\Hils[K]\otimes\Hils)\)
of~\(\Multunit\).
Namely, \(\Braiding{\Hils[L]}{\Hils[K]} \defeq Z\Flip\)
for the unique \(Z\in\U(\Hils[K]\otimes\Hils[L])\) with
\begin{equation}
  \label{eq:braid-Z-def}
  Z_{13} = \DuCorep{V}_{23} (\Corep{S}_{12})^* \DuCorep{V}_{23}^*
  \Corep{S}_{12}
  \qquad \text{in }\U(\Hils[K]\otimes \Hils\otimes \Hils[L]).
\end{equation}
The braiding in~\eqref{eq:T_corep} is the same as in the top-braided
pentagon equation for~\(\BrMultunit\).
Hence \((\Hils[L],\Corep{U},\BrMultunit)\)
is an example of such a right representation.

A morphism
\((\Hils[K]^1,\Corep{S}^1,\Corep{T}^1)\to
(\Hils[K]^2,\Corep{S}^2,\Corep{T}^2)\)
is a bounded operator \(a\colon \Hils[K]^1\to\Hils[K]^2\)
that intertwines both representations, that is,
\(a_1\circ \Corep{S}^1 = \Corep{S}^2\circ a_1\)
and \(a_1\circ \Corep{T}^1 = \Corep{T}^2\circ a_1\).
This turns the representations
of~\((\Multunit,\Corep{U},\DuCorep{V},\BrMultunit)\)
into a \(\Wst\)\nb-category
\(\Corepcat{\Multunit,\Corep{U},\DuCorep{V},\BrMultunit}\).
Forgetting both representations \(\Corep{S}\)
and~\(\Corep{T}\)
gives the forgetful functor to Hilbert spaces.  The functor
\(\Trivial\)
maps \(\Hils[K]\mapsto (\Hils[K],1,1)\).
If the identity map on~\(\Hils[K]\)
is an intertwiner
\((\Hils[K],\Corep{S}^1,\Corep{T}^1) \to
(\Hils[K],\Corep{S}^2,\Corep{T}^2)\),
then \(\Corep{S}^1=\Corep{S}^2\)
and \(\Corep{T}^1=\Corep{T}^2\).
So Assumption~\ref{assum:identity_arrow_equal} is satisfied.

We define a tensor product operation~\(\tenscorep\)
on \(\Corepcat{\Multunit,\Corep{U},\DuCorep{V},\BrMultunit}\) by
\[
(\Hils[K]^1,\Corep{S}^1,\Corep{T}^1) \tenscorep
(\Hils[K]^2,\Corep{S}^2,\Corep{T}^2)
\defeq
(\Hils[K]^1\otimes \Hils[K]^2,
\Corep{S}^1 \tenscorep \Corep{S}^2,
\Corep{T}^1 \tenscorep \Corep{T}^2)
\]
with
\begin{align*}
  \Corep{S}^1 \tenscorep \Corep{S}^2 &= \Corep{S}^1_{13}\Corep{S}^2_{23}
  \in \U(\Hils[K]^1\otimes\Hils[K]^2\otimes\Hils),\\
  \Corep{T}^1 \tenscorep \Corep{T}^2 &=
  (\Braiding{\Hils[L]}{\Hils[K]^2})^{\phantom1}_{23}
  \Corep{T}^1_{12} (\Dualbraiding{\Hils[K]^2}{\Hils[L]})^{\phantom1}_{23}
  \Corep{T}^2_{23}
  \in \U(\Hils[K]^1\otimes\Hils[K]^2\otimes\Hils[L]).
\end{align*}
The braiding operators \(\Braiding{\Hils[L]}{\Hils[K]^2}_{23}\)
and \(\Dualbraiding{\Hils[K]^2}{\Hils[L]}\)
use only the representations \(\Corep{S}\)
on~\(\Hils[K]^2\)
and~\(\DuCorep{V}\)
on~\(\Hils[L]\)
and therefore make sense.  In contrast,
\(\Braiding{\Hils[K]^2}{\Hils[L]}\)
and \(\Dualbraiding{\Hils[L]}{\Hils[K]^2}\)
would be defined if we had a left representation of~\(\Multunit\)
on~\(\Hils[K]^2\) instead of a right one.

\begin{lemma}
  \label{lem:corep_braided}
  The above definitions turn
  \(\Corepcat{\Multunit,\Corep{U},\DuCorep{V},\BrMultunit}\)
  into a Hilbert space tensor category.
\end{lemma}

\begin{proof}
  First, we ought to check that the tensor product above is
  well-defined, that is, gives representations again.  We check
  associativity of the tensor product first because we want to use it
  to prove that the tensor product is again a representation.  Let
  \(\Corep{S}^i \in \U(\Hils[K]^i\otimes\Hils)\)
  and \(\Corep{T}^i \in \U(\Hils[K]^i\otimes\Hils[L])\)
  for \(i=1,2,3\)
  be corepresentations of
  \((\Multunit,\Corep{U},\DuCorep{V},\BrMultunit)\).
  The definition of
  \(\Corep{T}\tenscorep \Corep{T}'\)
  makes sense for any \(\Multunit\)\nb-equivariant
  unitary operators \(\Corep{T},\Corep{T}'\).
  Thus both
  \((\Corep{T}^1\tenscorep \Corep{T}^2)\tenscorep \Corep{T}^3\)
  and \(\Corep{T}^1\tenscorep (\Corep{T}^2\tenscorep \Corep{T}^3)\)
  are defined even if we do not yet know that
  \(\Corep{T}^1\tenscorep \Corep{T}^2\)
  and \(\Corep{T}^2\tenscorep \Corep{T}^3\)
  give representations again.  We claim that both
  \((\Corep{T}^1\tenscorep \Corep{T}^2)\tenscorep \Corep{T}^3\)
  and \(\Corep{T}^1\tenscorep (\Corep{T}^2\tenscorep \Corep{T}^3)\)
  are equal to the \(\Multunit\)\nb-equivariant unitary
  \begin{equation}
    \label{eq:triple_tensor}
    (\Braiding{\Hils[L]}{\Hils[K]^2 \otimes \Hils[K]^3})_{234}
    \Corep{T}^1_{12}
    (\Dualbraiding{\Hils[K]^2 \otimes\Hils[K]^3}{\Hils[L]})^{\phantom1}_{234}
    (\Braiding{\Hils[L]}{\Hils[K]^3})_{34}
    \Corep{T}^2_{23}
    (\Dualbraiding{\Hils[K]^3}{\Hils[L]})^{\phantom1}_{34}
    \Corep{T}^3_{34}
  \end{equation}
  in \(\U(\Hils[K]^1 \otimes \Hils[K]^2 \otimes \Hils[K]^3 \otimes
  \Hils[L])\).  The operators
  \(\Braiding{\Hils[L]}{\Hils[K]^i}\colon \Hils[L]\otimes\Hils[K]^i
  \to\Hils[K]^i\otimes\Hils[L]\) are defined by
  \(\Braiding{\Hils[L]}{\Hils[K]^i} \defeq Z^i\Flip\), where
  \(Z^i\in\U(\Hils[K]^i\otimes\Hils[L])\)
  satisfies
  \begin{equation}
    \label{eq:braid-Z-def-1}
    Z^i_{13} = \DuCorep{V}_{23} (\Corep{S}^i_{12})^* \DuCorep{V}_{23}^*
    \Corep{S}^i_{12}
    \qquad \text{in }\U(\Hils[K]^i\otimes \Hils\otimes \Hils[L])
  \end{equation}
  for \(i=1,2,3\).  And
  \(\Braiding{\Hils[L]}{\Hils[K]^1\otimes\Hils[K]^2} =
  Z^{12}\Flip_{23}\Flip_{12}\), where
  \(Z^{12}\in\U(\Hils[K]^1\otimes\Hils[K]^2\otimes\Hils[L])\)
  satisfies
  \begin{equation}
    \label{eq:braid-Z-def-2}
    Z^{12}_{124} = \DuCorep{V}_{34} (\Corep{S}^1\tenscorep\Corep{S}^2)^*_{123}
    \DuCorep{V}_{34}^* (\Corep{S}^1\tenscorep\Corep{S}^2)_{123}
    \qquad \text{in }\U(\Hils[K]^1\otimes\Hils[K]^2\otimes
    \Hils\otimes \Hils[L]).
  \end{equation}
  This equation gives \(Z^{12}=Z^2_{23}Z^1_{13}\) when we plug in
  the definition of~\(\tenscorep\) and eliminate \(\Corep{S}^1\),
  \(\Corep{S}^2\) and~\(\DuCorep{V}\) using~\eqref{eq:braid-Z-def-1}.
  Therefore,
  \[
  \Braiding{\Hils[L]}{\Hils[K]^1\otimes\Hils[K]^2}
  = Z^{12}\Flip_{23}\Flip_{12}
  = Z^2_{23}Z^1_{13}\Flip_{23}\Flip_{12}
  = Z^2_{23}\Flip_{23}Z^1_{12}\Flip_{12}
  = \Braiding{\Hils[L]}{\Hils[K]^2}_{23}\Braiding{\Hils[L]}{\Hils[K]^1}_{12}.
  \]
  Similarly, \(\Braiding{\Hils[L]}{\Hils[K]^2\otimes\Hils[K]^3} =
  \Braiding{\Hils[L]}{\Hils[K]^3}_{23}\Braiding{\Hils[L]}{\Hils[K]^2}_{12}\).
  Now \(\Corep{T}^1\tenscorep
  (\Corep{T}^2\tenscorep\Corep{T}^3)\) and \((\Corep{T}^1\tenscorep
  \Corep{T}^2)\tenscorep\Corep{T}^3\) and the expression
  in~\eqref{eq:triple_tensor} are equal because they all simplify to
  \[
  \Braiding{\Hils[L]}{\Hils[K]^3}_{34}\Braiding{\Hils[L]}{\Hils[K]^2}_{23}
  \Corep{T}^1_{12}\Dualbraiding{\Hils[K]^2}{\Hils[L]}_{23}\Corep{T}^2_{23}
  \Dualbraiding{\Hils[K]^3}{\Hils[L]}_{34}\Corep{T}^3_{34}.
  \]

  Next, we check that the tensor product of two representations is
  again a representation.  The proof will also help to construct a
  natural right absorber later.  We claim that an operator
  \(\Corep{T}\in\U(\Hils[K]\otimes\Hils[L])\)
  together with \((\Hils[K],\Corep{S})\inOb\Corepcat{\Multunit}\)
  gives a representation if and only if~\(\Corep{T}\) is an intertwiner
  \[
  (\Hils[K] \otimes \Hils[L],
  \Corep{S}\tenscorep\Corep{U},
  \Corep{T}\tenscorep\Corep{F}) \to
  (\Hils[K] \otimes \Hils[L],
  \Corep{S}\tenscorep\Corep{U},
  1\tenscorep\Corep{F}).
  \]
  Indeed, being such an intertwiner means being equivariant with
  respect to \(\Corep{S}\tenscorep\Corep{U}\)
  and intertwining
  \(\Corep{T}\tenscorep\Corep{F} =
  (\Braiding{\Hils[L]}{\Hils[K]^2})_{23} \Corep{T}_{12}
  (\Dualbraiding{\Hils[K]^2}{\Hils[L]})^{\phantom1}_{23}
  \Corep{F}_{23}\)
  with \(1\tenscorep\Corep{F} = \Corep{F}_{23}\).  The latter is
  exactly our representation condition.  Assume that
  \(\Corep{T}^1\in\U(\Hils[K]^1\otimes\Hils[L])\)
  and \(\Corep{T}^2\in\U(\Hils[K]^2\otimes\Hils[L])\) are braided
  representations.
  Since~\(\Corep{T}^2_{23}\)
  is equivariant, when we conjugate it with the braiding operator
  \((\Braiding{\Hils[L]}{\Hils[K]^2 \otimes \Hils[L]})_{234}\)
  on \(\Hils[K]^1\otimes\Hils[K]^2\otimes\Hils[L]\otimes\Hils[L]\),
  then we merely transfer it to~\(\Corep{T}^2_{34}\),
  which commutes with~\(\Corep{T}^1_{12}\).
  Thus~\eqref{eq:triple_tensor} shows that~\(\Corep{T}^2_{23}\)
  is also an intertwiner
  \[
  (\Hils[K]^1 \otimes \Hils[K]^2 \otimes \Hils[L],
  \Corep{S}^1\tenscorep \Corep{S}^2\tenscorep \Corep{U},
  \Corep{T}^1\tenscorep \Corep{T}^2 \tenscorep \Corep{F}) \to
  (\Hils[K]^1 \otimes \Hils[K]^2 \otimes \Hils[L],
  \Corep{S}^1\tenscorep \Corep{S}^2\tenscorep \Corep{U},
  \Corep{T}^1\tenscorep 1 \tenscorep \Corep{F}).
  \]
  Similarly, the braiding operator
  \(\Dualbraiding{\Hils[K]^2}{\Hils[L]}\) gives an intertwiner
  \begin{equation}
    \label{eq:braiding_intertwiner}
    (\Hils[K]^1 \otimes \Hils[K]^2 \otimes \Hils[L],
    \Corep{S}^1\tenscorep \Corep{S}^2\tenscorep \Corep{U},
    \Corep{T}^1\tenscorep 1 \tenscorep \Corep{F})
    \xrightarrow{\Dualbraiding{\Hils[K]^2}{\Hils[L]}_{23}}
    (\Hils[K]^1 \otimes \Hils[L] \otimes \Hils[K]^2,
    \Corep{S}^1\tenscorep \Corep{U} \tenscorep \Corep{S}^2,
    \Corep{T}^1 \tenscorep \Corep{F} \tenscorep 1).
  \end{equation}
  Now the operator~\(\Corep{T}^1_{12}\) is an intertwiner
  \[
  (\Hils[K]^1 \otimes \Hils[L] \otimes \Hils[K]^2,
  \Corep{S}^1\tenscorep \Corep{U} \tenscorep \Corep{S}^2,
  \Corep{T}^1 \tenscorep \Corep{F} \tenscorep 1)\to
  (\Hils[K]^1 \otimes \Hils[L] \otimes \Hils[K]^2,
  \Corep{S}^1\tenscorep \Corep{U} \tenscorep \Corep{S}^2,
  1 \tenscorep \Corep{F} \tenscorep 1).
  \]
  The unitary \(\Braiding{\Hils[L]}{\Hils[K]^2}\)
  gives an intertwiner from the last representation back to
  \[
  (\Hils[K]^1 \otimes \Hils[K]^2 \otimes \Hils[L],
  \Corep{S}^1 \tenscorep \Corep{S}^2 \tenscorep \Corep{U},
  1 \tenscorep 1 \tenscorep \Corep{F}).
  \]
  Hence \(\Corep{T}^1\tenscorep \Corep{T}^2\)
  has the expected intertwining property to be a representation.

  Now we check Assumption~\ref{assum:arrow_cancellative}.  Let
  \(a\in\Bound(\Hils[K]^1)\)
  be such that \(a_1\in\U(\Hils[K]^1\otimes\Hils[K]^2)\)
  is an intertwiner for the tensor product representation.
  Since~\(a\)
  commutes with~\(\Corep{S}^2_{23}\),
  the equivariance with respect to
  \(\Corep{S}^1\tenscorep \Corep{S}^2\)
  gives that~\(a\)
  is \(\Corep{S}^1\)-equivariant.
  Since~\(a_1\)
  commutes with \(\Corep{T}^1\tenscorep\Corep{T}^2\),
  \((\Braiding{\Hils[L]}{\Hils[K]^2 \otimes \Hils[K]^3})_{234}\),
  and \(\Corep{T}^2_{23}\),
  it follows that~\(a\)
  is an intertwiner for~\(\Corep{T}^1_{12}\) as well.

  Finally, the functors \(\Forget\)
  and~\(\Trivial\)
  are strict tensor functors by definition, and~\(\Trivial(\C)\)
  with the canonical unit transformations is indeed a tensor unit.
\end{proof}

\begin{proposition}
  \label{pro:braided_absorber}
  The representation
  \[
  \rho = (\Hils\otimes\Hils[L],
  \Multunit\tenscorep\Corep{U},
  1\tenscorep\BrMultunit)
  \]
  is a natural right absorber for the tensor category
  \(\Corepcat{\Multunit,\Corep{U},\DuCorep{V},\BrMultunit}\).
\end{proposition}

\begin{proof}
  We must construct an intertwiner
  \(A^x\colon x\otimes \rho \to \Trivial(x)\otimes\rho\)
  for any representation \(x=(\Hils[K],\Corep{S},\Corep{T})\)
  of \((\Multunit,\Corep{U},\DuCorep{V},\BrMultunit)\).
  We claim that the composite operator
  \begin{multline*}
    (\Hils[K]\otimes \Hils\otimes\Hils[L],
    \Corep{S} \tenscorep \Multunit \tenscorep \Corep{U},
    \Corep{T} \tenscorep 1_{\Hils} \tenscorep \BrMultunit)
    \xrightarrow{\Corep{T}\tenscorep 1_{\Hils}}
    (\Hils[K]\otimes \Hils\otimes\Hils[L],
    \Corep{S}\tenscorep \Multunit \tenscorep \Corep{U},
    1_{\Hils[K]} \tenscorep 1_{\Hils} \tenscorep \BrMultunit)
    \\\xrightarrow{\Corep{S}_{12}}
    (\Hils[K]\otimes \Hils\otimes\Hils[L],
    1_{\Hils[K]}\tenscorep \Multunit \tenscorep \Corep{U},
    1_{\Hils[K]} \tenscorep 1_{\Hils} \tenscorep \BrMultunit)
  \end{multline*}
  has the properties required in
  Definition~\ref{def:natural_absorber}.  The triple
  \((\Hils[K],\Corep{S},1)\)
  is a representation of
  \((\Multunit,\Corep{U},\DuCorep{V},\BrMultunit)\)
  for any right representation~\(\Corep{S}\)
  of~\(\Multunit\),
  and a map between representations of this form is an intertwiner
  if and only if it is an intertwiner for the representations
  of~\(\Multunit\).
  In particular, the second map~\(\Corep{S}_{12}\)
  above is an intertwiner, see Example~\ref{exa:Multunit_absorber}.
  Moreover, since there are representations \((\Hils,\Multunit,1)\)
  and
  \(x \otimes (\Hils,\Multunit,1) =
  (\Hils[K]\otimes\Hils,\Corep{S}\tenscorep\Multunit,
  \BrMultunit\tenscorep 1)\),
  the map \(\Corep{T}\tenscorep 1_{\Hils}\)
  above is an intertwiner as well, see the proof of
  Lemma~\ref{lem:corep_braided}.  Thus the composite map is an
  intertwiner \(x\otimes\rho \to \Trivial(x)\otimes\rho\)
  as needed.  These two operators and their composite are natural by
  construction, that is, \ref{en:natural_absorber1} holds.
  We check condition \ref{en:natural_absorber2}.  Let
  \((\Hils[K]^i,\Corep{S}^i,\Corep{T}^i)\) be representations of
  \((\Multunit,\Corep{U},\DuCorep{V},\BrMultunit)\).  We shall use
  the diagram in Figure~\ref{fig:braided_corep_absorber}.
  \begin{figure}[htbp]
    \centering
    \begin{tikzpicture}
      \matrix (m) [cd,column sep=3em, row sep=8ex] {
        \tworow{\Corep{S}^1}{\Corep{T}^1} \tworow{\Corep{S}^2}{\Corep{T}^2} \tworow{\Corep{U}}{\BrMultunit} \tworow{\Multunit}{1}&
        \tworow{\Corep{S}^1}{\Corep{T}^1} \tworow{\Corep{S}^2}{1} \tworow{\Corep{U}}{\BrMultunit} \tworow{\Multunit}{1}&
        \tworow{\Corep{S}^1}{\Corep{T}^1} \tworow{\Corep{U}}{\BrMultunit} \tworow{\Corep{S}^2}{1} \tworow{\Multunit}{1}&
        \tworow{\Corep{S}^1}{1} \tworow{\Corep{U}}{\BrMultunit} \tworow{\Corep{S}^2}{1} \tworow{\Multunit}{1}&
        \tworow{\Corep{S}^1}{1} \tworow{\Corep{S}^2}{1} \tworow{\Corep{U}}{\BrMultunit} \tworow{\Multunit}{1}\\
        \tworow{\Corep{S}^1}{\Corep{T}^1} \tworow{\Corep{S}^2}{\Corep{T}^2} \tworow{\Multunit}{1} \tworow{\Corep{U}}{\BrMultunit}&
        \tworow{\Corep{S}^1}{\Corep{T}^1} \tworow{\Corep{S}^2}{1} \tworow{\Multunit}{1} \tworow{\Corep{U}}{\BrMultunit}&
        \tworow{\Corep{S}^1}{\Corep{T}^1} \tworow{\Corep{U}}{\BrMultunit} \tworow{1}{1} \tworow{\Multunit}{1}&
        \tworow{\Corep{S}^1}{1} \tworow{\Corep{U}}{\BrMultunit} \tworow{1}{1} \tworow{\Multunit}{1}&
        \tworow{\Corep{S}^1}{1} \tworow{\Corep{S}^2}{1} \tworow{\Multunit}{1} \tworow{\Corep{U}}{\BrMultunit}\\
        &\tworow{\Corep{S}^1}{\Corep{T}^1} \tworow{1}{1} \tworow{\Multunit}{1} \tworow{\Corep{U}}{\BrMultunit}&
        \tworow{\Corep{S}^1}{\Corep{T}^1} \tworow{1}{1} \tworow{\Corep{U}}{\BrMultunit} \tworow{\Multunit}{1}&
        \tworow{\Corep{S}^1}{1} \tworow{1}{1} \tworow{\Corep{U}}{\BrMultunit} \tworow{\Multunit}{1}&
        \tworow{\Corep{S}^1}{1} \tworow{1}{1} \tworow{\Multunit}{1} \tworow{\Corep{U}}{\BrMultunit}\\
        &&&&\tworow{1}{1} \tworow{1}{1} \tworow{\Multunit}{1} \tworow{\Corep{U}}{\BrMultunit}\\
      };
      \labelar{m-1-1}{\Corep{T}^2_{23}}{}{m-1-2};
      \labelar{m-1-2}{\Dualbraiding{\Hils[K]^2}{\Hils[L]}}{}{m-1-3};
      \labelar{m-1-3}{\Corep{T}^1_{12}}{}{m-1-4};
      \labelar{m-1-4}{\Braiding{\Hils[L]}{\Hils[K]^2}}{}{m-1-5};
      \labelar{m-1-2}{}{\Braiding{\Hils[L]}{\Hils[H]}}{m-2-2};
      \labelar{m-1-3}{\Corep{S}^2_{34}}{}{m-2-3};
      \labelar{m-1-4}{\Corep{S}^2_{34}}{}{m-2-4};
      \labelar{m-1-4}{\Braiding{\Hils[L]}{\Hils[K]^2\Hils[H]}}{}{m-2-5};
      \labelar{m-1-5}{\Braiding{\Hils[L]}{\Hils[H]}}{}{m-2-5};
      \labelar{m-2-1}{\Dualbraiding{\Hils[H]}{\Hils[L]}}{}{m-1-1};
      \labelar{m-2-2}{\Dualbraiding{\Hils[K]^2\Hils[H]}{\Hils[L]}}{}{m-1-3};
      \labelar{m-2-3}{\Corep{T}^1_{12}}{}{m-2-4};
      \labelar{m-2-1}{}{A^2_{234}}{m-3-2};
      \labelar{m-2-2}{}{\Corep{S}^2_{23}}{m-3-2};
      \labelar{m-2-4}{\Braiding{\Hils[L]}{\Hils[K]^2\Hils[H]}}{}{m-3-5};
      \labelar{m-2-5}{\Corep{S}^2_{23}}{}{m-3-5};
      \labelar{m-3-2}{\Dualbraiding{\Hils[K]^2\Hils[H]}{\Hils[L]}}{}{m-2-3};
      \labelar{m-3-3}{\Dualbraiding{\Hils[K]^2}{\Hils[L]}}{}{m-2-3};
      \labelar{m-3-4}{\Dualbraiding{\Hils[K]^2}{\Hils[L]}}{}{m-2-4};
      \labelar{m-3-2}{\Dualbraiding{\Hils[H]}{\Hils[L]}}{}{m-3-3};
      \labelar{m-3-3}{\Corep{T}^1_{13}}{}{m-3-4};
      \labelar{m-3-4}{\Braiding{\Hils[L]}{\Hils[H]}}{}{m-3-5};
      \labelar{m-3-2}{}{A^1_{134}}{m-4-5};
      \labelar{m-3-5}{\Corep{S}^1_{13}}{}{m-4-5};
    \end{tikzpicture}
    \caption{Commuting diagram in
      \(\Corepcat{\Multunit,\Corep{U},\DuCorep{V},\BrMultunit}\)
      that proves the condition~\ref{en:natural_absorber2}.}
    \label{fig:braided_corep_absorber}
  \end{figure}
  This diagram uses short-hand notation for
  representations.  For instance,
  \(\tworow{\Corep{S}^1}{\Corep{T}^1}
  \tworow{\Corep{S}^2}{\Corep{T}^2} \tworow{\Corep{U}}{\BrMultunit}
  \tworow{\Multunit}{1}\) denotes the representation
  \[
  (\Hils[K]^1\otimes\Hils[K]^2\otimes\Hils[L]\otimes\Hils[H],
  \Corep{S}^1\tenscorep \Corep{S}^2\tenscorep \Corep{U}\tenscorep \Multunit,
  \Corep{T}^1\tenscorep \Corep{T}^2\tenscorep \BrMultunit\tenscorep 1).
  \]
  All braiding operators in this diagram exist because the Hilbert
  space~\(\Hils[L]\)
  is on the top strand.  They are intertwiners of braided
  representations, compare~\eqref{eq:braiding_intertwiner}.  The
  remaining arrows are also intertwiners of braided representations by
  the proof of Lemma~\ref{lem:corep_braided}.  Before we show that the
  diagram in Figure~\ref{fig:braided_corep_absorber} commutes, we
  deduce the condition \ref{en:natural_absorber2} from it.  The
  arrow from the \((2,1)\)-entry
  to the \((2,5)\)-entry
  in Figure~\ref{fig:braided_corep_absorber} along the top boundary is
  the intertwiner
  \begin{multline*}
    \Corep{T}^1\tenscorep \Corep{T}^2\tenscorep 1_{\Hils[H]}\colon
    (\Hils[K]^1\otimes\Hils[K]^2\otimes\Hils\otimes\Hils[L],
    \Corep{S}^1\tenscorep \Corep{S}^2\tenscorep \Multunit\tenscorep \Corep{U},
    \Corep{T}^1\tenscorep \Corep{T}^2\tenscorep 1\tenscorep \BrMultunit)
    \\\to
    (\Hils[K]^1\otimes\Hils[K]^2\otimes\Hils\otimes\Hils[L],
    \Corep{S}^1\tenscorep \Corep{S}^2\tenscorep \Multunit\tenscorep \Corep{U},
    1\tenscorep 1\tenscorep 1\tenscorep F),
  \end{multline*}
  compare the proof in Lemma~\ref{lem:corep_braided} that the tensor
  product is associative.  And the arrow going downward from there is
  \((\Corep{S}^1 \tenscorep \Corep{S}^2)_{123}\).
  So the composite arrow is the absorbing intertwiner for the tensor
  product
  \((\Hils[K]^1\otimes\Hils[K]^2, \Corep{S}^1\tenscorep \Corep{S}^2,
  \Corep{T}^1\tenscorep \Corep{T}^2)\).
  Similarly, the arrows labeled \(A^2\)
  and~\(A^1\)
  are the absorbing intertwiners for
  \((\Hils[K]^2,\Corep{S}^2,\Corep{T}^2)\) and
  \((\Hils[K]^1,\Corep{S}^1,\Corep{T}^1)\), respectively.  Hence the
  commutativity of the boundary of the diagram in
  Figure~\ref{fig:braided_corep_absorber} is exactly
  \ref{en:natural_absorber2}.

  Now we check that the diagram in
  Figure~\ref{fig:braided_corep_absorber} commutes.
  The four triangles of braiding
  operators commute
  because the braiding operators have enough of the properties of a
  braided monoidal category, compare the proof in
  Lemma~\ref{lem:corep_braided}.  The two pentagons with \(A^1\)
  and~\(A^2\) as one of the faces commute by definition of our
  absorbing intertwiners.  The two parallelograms
  with~\(\Corep{S}^2\) and braiding operators commute because the
  braiding operators are natural with respect to intertwiners of
  \(\Multunit\)\nb-representations.  The square with
  \(\Corep{T}^1_{12}\) and~\(\Corep{S}^2_{34}\) commutes because we
  operate on different legs.  Finally, we consider the square
  involving~\(\Corep{T}^1\) and the braiding operator
  \(\Dualbraiding{\Hils[K]^2}{\Hils[L]}\).  Here~\(\Hils[K]^2\)
  carries the trivial representation of~\(\Multunit\), so
  that the braiding is just the tensor flip~\(\Sigma_{23}\).  Thus
  the square commutes, and now we have seen that the entire diagram
  commutes.
\end{proof}

\begin{theorem}
  \label{the:braided_to_multunit}
  The operator
  \[
  \Multunit[C] \defeq \Multunit_{13} \Corep{U}_{23} \DuCorep{V}_{34}^*
  \BrMultunit_{24} \DuCorep{V}_{34}
  \in\U(\Hils\otimes\Hils[L]\otimes\Hils\otimes\Hils[L])
  \]
  is a multiplicative unitary such that there is a fully faithful,
  strict tensor functor
  \(\Phi\colon \Corepcat{\Multunit,\Corep{U},\DuCorep{V},\BrMultunit}
  \to \Corepcat{\Multunit[C]}\)
  with \(\Forget\circ \Phi = \Forget\).
  The functor~\(\Phi\)
  maps a representation \((\Hils[K],\Corep{S},\Corep{T})\)
  of \((\Multunit,\Corep{U},\DuCorep{V},\BrMultunit)\)
  to the following representation of~\(\Multunit[C]\):
  \[
  \Corep{S}_{12} (\Corep{T}\tenscorep1_{\Hils})
  = \Corep{S}_{12} \DuCorep{V}_{23}^* \Corep{T}_{13} \DuCorep{V}_{23}
  \in \U(\Hils[K]\otimes\Hils\otimes\Hils[L]).
  \]
  The functor~\(\Phi\)
  is an isomorphism of categories if \(\Multunit[C]\)
  and~\(\Multunit\) are manageable.
\end{theorem}

The manageability of~\(\Multunit[C]\) is expressed
in~\cite{Meyer-Roy-Woronowicz:Qgrp_proj} in terms of the braided
multiplicative unitary
\((\Multunit,\Corep{U},\DuCorep{V},\BrMultunit)\).

\begin{proof}
  We have found a natural right absorber \((\rho,A)\)
  in Proposition~\ref{pro:braided_absorber}.
  Proposition~\ref{pro:absorber} shows that~\(A^\rho\)
  is a multiplicative unitary and that \(x\mapsto A^x\)
  is a fully faithful, strict tensor functor
  \(\Corepcat{\Multunit,\Corep{U},\DuCorep{V},\BrMultunit} \to
  \Corepcat{A^\rho}\).
  By definition,
  \(A^x = \Corep{S}_{12} (\Corep{T}\tenscorep 1_{\Hils}) =
  \Corep{S}_{12} (\Braiding{\Hils[L]}{\Hils})^{\phantom1}_{23}
  \Corep{T}_{12} (\Dualbraiding{\Hils}{\Hils[L]})^{\phantom1}_{23}\)
  and, in particular,
  \[
  A^\rho = (\Multunit\tenscorep\Corep{U})_{123}
  (1_{\Hils} \tenscorep \BrMultunit\tenscorep 1_{\Hils})
  = \Multunit_{13}\Corep{U}_{23}
  (\Braiding{\Hils[L]}{\Hils})_{34}
  \BrMultunit_{23}
  (\Dualbraiding{\Hils}{\Hils[L]})^{\phantom1}_{34}.
  \]
  The braiding unitary
  \(\Braiding{\Hils[L]}{\Hils} \in \U(\Hils[L]\otimes
  \Hils,\Hils\otimes \Hils[L])\)
  is equal to~\(Z\Sigma\)
  for the unique unitary \(Z\in \U(\Hils\otimes \Hils[L])\)
  that satisfies
  \[
  Z_{13} = \DuCorep{V}_{23} \Multunit[*]_{12} \DuCorep{V}_{23}^*
  \Multunit_{12}
  \qquad \text{in }\U(\Hils\otimes \Hils\otimes \Hils[L]),
  \]
  compare~\eqref{eq:braid-Z-def} and
  \cite{Meyer-Roy-Woronowicz:Twisted_tensor_2}*{(6.10)}.  We have
  \(\DuCorep{V}_{23} \Multunit_{12} \DuCorep{V}_{23}^* =
  \Multunit_{12} \DuCorep{V}_{13}\)
  because \(\DuCorep{V}\)
  is a left representation of~\(\Multunit\).
  Hence \(Z_{13} = \DuCorep{V}_{13}^*\).
  Since
  \(\Sigma_{34} \BrMultunit_{23} \Sigma_{34} = \BrMultunit_{24}\),
  we get the asserted formulas for \(A^x\)
  and~\(A^\rho\).
  We still have to prove that every representation of~\(\Multunit[C]\)
  comes from one of \((\Multunit,\Corep{U},\DuCorep{V},\BrMultunit)\).
  This will take a while and require some further results.  This proof
  will be completed at the end of this article.
\end{proof}

\begin{proposition}
  \label{pro:projection_on_C}
  The operators
  \(\Multunit_{13}\Corep{U}_{23} \in
  \U(\Hils\otimes\Hils[L]\otimes\Hils)\)
  and \(\Multunit_{12} \in\U(\Hils\otimes\Hils\otimes\Hils[L])\)
  are bicharacters from the multiplicative unitary~\(\Multunit[C]\)
  to \(\Multunit\in\U(\Hils\otimes\Hils)\)
  and back, whose composite from~\(\Multunit[C]\)
  to itself is equal to
  \(\ProjBichar \defeq \Multunit_{13} \Corep{U}_{23}
  \in\U(\Hils\otimes\Hils[L]\otimes\Hils\otimes\Hils[L])\).
  Equivalently, the following pentagon-like equations hold:
  \begin{equation}
    \label{eq:projection_pentagon}
    \ProjBichar_{23}\Multunit[C]_{12}
    = \Multunit[C]_{12}\ProjBichar_{13}\ProjBichar_{23},\qquad
    \Multunit[C]_{23} \ProjBichar_{12}
    = \ProjBichar_{12}\ProjBichar_{13}\Multunit[C]_{23},\qquad
    \ProjBichar_{23}\ProjBichar_{12} =
    \ProjBichar_{12}\ProjBichar_{13}\ProjBichar_{23}.
  \end{equation}
\end{proposition}

\begin{proof}
  There are two obvious strict tensor functors between the Hilbert
  space tensor categories
  \(\Corepcat{\Multunit,\Corep{U},\DuCorep{V},\BrMultunit}\)
  and \(\Corepcat{\Multunit}\), namely, the forgetful functor
  \[
  \Corepcat{\Multunit,\Corep{U},\DuCorep{V},\BrMultunit} \to
  \Corepcat{\Multunit},\qquad
  (\Hils[K],\Corep{S},\Corep{T}) \mapsto (\Hils[K],\Corep{S}),
  \]
  and the functor
  \[
  \Corepcat{\Multunit} \to
  \Corepcat{\Multunit,\Corep{U},\DuCorep{V},\BrMultunit},\qquad
  (\Hils[K],\Corep{S}) \mapsto (\Hils[K],\Corep{S},1_{\Hils[K]}).
  \]
  The definitions imply immediately that these are strict tensor
  functors that are compatible with the forgetful functors
  to~\(\Hilb\).
  Both tensor categories involved have natural right absorbers, and
  the associated multiplicative unitaries are \(\Multunit[C]\)
  and~\(\Multunit\),
  respectively.  Proposition~\ref{pro:functor_Corep_morphism} produces
  bicharacters from strict tensor functors like the ones above.
  Furthermore, the composite functor on~\(\Corepcat{\Multunit}\)
  is the identity.  Correspondingly, the composite bicharacter
  from~\(\Multunit\)
  to itself is the bicharacter that describes the identity functor,
  which is~\(\Multunit\)
  itself.  And the composite bicharacter from~\(\Multunit[C]\)
  to itself is idempotent, which means that it satisfies the pentagon
  equation.  It remains to compute the bicharacters that we
  get from the formulas in
  Proposition~\ref{pro:functor_Corep_morphism}.

  The bicharacter describing the functor
  \(\Corepcat{\Multunit,\Corep{U},\DuCorep{V},\BrMultunit} \to
  \Corepcat{\Multunit}\) is the canonical unitary intertwiner
  \[
  \Multunit\tenscorep\Corep{U} = \Multunit_{13}\Corep{U}_{23}\colon
  (\Hils\otimes\Hils[L], \Multunit\tenscorep\Corep{U}) \otimes
  (\Hils,\Multunit) \to
  (\Hils\otimes\Hils[L], 1) \otimes
  (\Hils,\Multunit),
  \]
  that is, we get
  \(\Multunit_{13}\Corep{U}_{23}\in
  \U(\Hils\otimes\Hils[L]\otimes\Hils)\).

  The bicharacter describing the functor
  \(\Corepcat{\Multunit} \to
  \Corepcat{\Multunit,\Corep{U},\DuCorep{V},\BrMultunit}\)
  is the natural isomorphism
  \[
  (\Hils,\Multunit,1) \otimes
  (\Hils\otimes\Hils[L],
  \Multunit\otimes\Corep{U},
  1\tenscorep\BrMultunit)
  \to
  (\Hils,1,1) \otimes
  (\Hils\otimes\Hils[L],
  \Multunit\otimes\Corep{U},
  1\tenscorep\BrMultunit)
  \]
  described during the proof of
  Proposition~\ref{pro:braided_absorber}.  Since the representation
  of~\(\BrMultunit\)
  is~\(1\) here,
  this simplifies to the unitary
  \(\Multunit_{12}\in \U(\Hils\otimes\Hils\otimes\Hils[L])\).

  By the definition of the composition of bicharacters in
  \cite{Meyer-Roy-Woronowicz:Homomorphisms}*{Definition 3.5}, the
  composite bicharacter from~\(\Multunit[C]\)
  to itself is~\(\Multunit_{13}\Corep{U}_{23}\)
  if and only if the following equation holds
  in
  \(\U(\Hils\otimes\Hils[L]\otimes\Hils\otimes\Hils\otimes\Hils[L])\):
  \[
  \Multunit_{34} (\Multunit_{13} \Corep{U}_{23})
  = (\Multunit_{13} \Corep{U}_{23})
  (\Multunit_{14} \Corep{U}_{24}) \Multunit_{34}
  \]
  Indeed, the representation property of~\(\Corep{U}\)
  and the pentagon equation for~\(\Multunit\) give
  \[
  \Multunit_{13} \Corep{U}_{23} \Multunit_{14} \Corep{U}_{24}
  \Multunit_{34}
  = \Multunit_{13} \Multunit_{14} \Corep{U}_{23} \Corep{U}_{24}
  \Multunit_{34}
  = \Multunit_{13} \Multunit_{14} \Multunit_{34} \Corep{U}_{23}
  = \Multunit_{34} \Multunit_{13} \Corep{U}_{23}
  \]
  as desired.  The general theory says that the unitaries
  in~\eqref{eq:projection_pentagon} are
  bicharacters and that the bicharacter~\(\ProjBichar\) is idempotent,
  that is, satisfies the pentagon equation.
\end{proof}

It remains to prove that every representation of~\(\Multunit[C]\)
comes from a representation of the braided multiplicative unitary
if~\(\Multunit[C]\)
is manageable.  That is, we want it to be of the form
\(\Corep{S}_{12} \DuCorep{V}_{23}^* \Corep{T}_{13} \DuCorep{V}_{23}\)
for some representation \((\Hils[K],\Corep{S},\Corep{T})\)
of \((\Multunit,\Corep{U},\DuCorep{V},\BrMultunit)\).
So we start with a representation \((\Hils[K],\Corep{A})\)
of~\(\Multunit[C]\).
The Hilbert space must remain~\(\Hils[K]\).
We have described the functor
\[
\Corepcat{\Multunit,\Corep{U},\DuCorep{V},\BrMultunit} \to
\Corepcat{\Multunit},\qquad
(\Hils[K],\Corep{S},\Corep{T})\mapsto (\Hils[K],\Corep{S}),
\]
through the bicharacter \(\Multunit_{13} \Corep{U}_{23}\)
from~\(\Multunit[C]\)
to~\(\Multunit\)
in Proposition~\ref{pro:projection_on_C}.  The proof of
Proposition~\ref{pro:functor_Corep_morphism} shows that there is a
unique unitary \(\Corep{S}\in\U(\Hils[K]\otimes\Hils)\) with
\begin{equation}
  \label{eq:induced_corep}
  (\Multunit_{24}\Corep{U}_{34})\Corep{A}_{123}
  = \Corep{A}_{123} \Corep{S}_{14} (\Multunit_{24}\Corep{U}_{34})
  \in \U(\Hils[K]\otimes \Hils \otimes \Hils[L] \otimes \Hils)
\end{equation}
because the multiplicative unitary~\(\Multunit\)
is manageable: this is the functor on representation categories
induced by the bicharacter \(\Multunit_{13} \Corep{U}_{23}\).
Now~\(\Corep{T}\)
should satisfy
\(\Corep{A}_{123} = \Corep{S}_{12} \DuCorep{V}_{23}^* \Corep{T}_{13}
\DuCorep{V}_{23}\), that is,
\[
\Corep{T}_{13} =
\DuCorep{V}_{23} \Corep{S}_{12}^* \Corep{A}_{123} \DuCorep{V}_{23}^*
\qquad\text{in }\U(\Hils[K]\otimes \Hils \otimes \Hils[L]).
\]
It remains to prove, first, that the right hand side has trivial
second leg, so that it comes from a unitary
\(\Corep{T}\in\U(\Hils[K]\otimes\Hils[L])\);
and, secondly, that \((\Hils[K],\Corep{S},\Corep{T})\)
is a representation of
\((\Multunit,\Corep{U},\DuCorep{V},\BrMultunit)\).
Since these computations are quite unpleasant, we proceed indirectly.
During this proof, we say that a
representation
of~\(\Multunit[C]\)
\emph{comes from a braided representation} if it belongs to the image
of the functor
\(\Corepcat{\Multunit,\Corep{U},\DuCorep{V},\BrMultunit}\to
\Corepcat{\Multunit[C]}\).

\begin{lemma}
  \label{lem:lift_cancel}
  Let \((\Hils[K]^1,\Corep{A}^1)\)
  and \((\Hils[K]^2,\Corep{A}^2)\)
  be representations of~\(\Multunit[C]\).
  If \((\Hils[K]^1,\Corep{A}^1)\)
  and
  \((\Hils[K]^1\otimes \Hils[K]^2,\Corep{A}^1\tenscorep \Corep{A}^2)\)
  come from braided representations, then so does
  \((\Hils[K]^2,\Corep{A}^2)\).
\end{lemma}

\begin{proof}
  We define~\(\Corep{S}^i\)
  and~\(\Corep{T}^i\)
  for \(i=1,2\)
  as above.  We know that \((\Hils[K]^1,\Corep{S}^1,\Corep{T}^1)\)
  is a braided representation.  But at first, we only know
  \(\Corep{T}^2 \in \U(\Hils[K]^2\otimes \Hils[L]\otimes\Hils)\).
  We may, nevertheless, recycle the diagram in
  Figure~\ref{fig:braided_corep_absorber}, treating it as a diagram in
  \(\Corepcat{\Multunit}\)
  only, and replacing the top left arrow \(\Corep{T}^2_{23}\)
  by \(\Corep{T}^2_{234}\).
  The two pentagons still commute by definition of
  \(\Corep{S}^i,\Corep{T}^i\).
  The four triangles of braiding operators in
  Figure~\ref{fig:braided_corep_absorber} commute as before.  So do
  the parallelograms containing \(\Corep{S}^2_{23}\)
  and braiding operators, and the two squares in the middle: this only
  needs \((\Corep{S}^1,\Corep{T}^1)\)
  to be a braided representation, which we have assumed.  Hence the
  entire diagram commutes.  The composite arrow from the
  \((2,1)\)-entry
  to the \((5,4)\)-entry
  is the tensor product representation
  \(\Corep{A}^1 \tenscorep \Corep{A}^2\).
  We have assumed that this comes from a braided representation.  This
  must be of the form
  \((\Corep{S}^1\tenscorep \Corep{S}^2,\Corep{T})\)
  for some
  \(\Corep{T}\in \U(\Hils[K]^1\otimes\Hils[K]^2\otimes\Hils[L])\).
  Hence
  \[
  (\Dualbraiding{\Hils[K]^2}{\Hils[L]})_{23}
  \Corep{T}^1_{12}
  (\Braiding{\Hils[L]}{\Hils[K]^2})_{23}
  \Corep{T}^2_{234}
  = \Corep{T}_{123}.
  \]
  Therefore, \(\Corep{T}^2_{234}\)
  acts trivially on the fourth leg.  So
  \(\Corep{A}^2 = \Corep{S}^2_{12} \DuCorep{V}_{23}^* \Corep{T}^2_{13}
  \DuCorep{V}_{23}\)
  for some \(\Corep{T}^2\in\U(\Hils[K]^2\otimes\Hils[L])\).
  In the proof of Lemma~\ref{lem:corep_braided}, we have shown that a
  unitary~\(\Corep{T}\)
  in \(\U(\Hils[K]\otimes\Hils[L])\)
  together with a representation~\((\Hils[K],\Corep{S})\)
  of~\(\Multunit\)
  is a braided representation if and only if \(\Corep{T}\)
  is an intertwiner from \(\Corep{T}\tenscorep \BrMultunit\)
  to \(1_{\Hils[K]}\tenscorep \BrMultunit\).
  Therefore,
  \(\Corep{T}^1\tenscorep \Corep{T}^2=
  (\Braiding{\Hils[L]}{\Hils[K]^2})^{\phantom1}_{23} \Corep{T}^1_{12}
  (\Dualbraiding{\Hils[K]^2}{\Hils[L]})^{\phantom1}_{23}
  \Corep{T}^2_{23}\)
  and \(\Corep{T}^1\)
  are intertwiners of braided representations.  So are the braiding
  operators, compare~\eqref{eq:braiding_intertwiner}.  Hence
  \(\Corep{T}^2_{23}\in\U(\Hils[K]^1\otimes\Hils[K]^2\otimes\Hils[L])\)
  is an intertwiner of braided representations.  Then so
  is~\(\Corep{T}^2\)
  itself.  This means that~\((\Corep{S}^2,\Corep{T}^2)\)
  is a braided representation.
\end{proof}

Since~\(\Multunit[C]\)
is manageable, Proposition~\ref{pro:left_right_absorber} shows
that~\(\Corepcat{\Multunit[C]}\)
contains a (natural) left absorber~\(\Corep{A}^1\).
Even more, the proof shows that we may choose~\(\Corep{A}^1\)
to be isomorphic to a direct sum of copies of~\(\Multunit[C]\).
By definition, \(\Multunit[C]\)
comes from the braided representation
\((\Hils\otimes\Hils[L], \Multunit \tenscorep \Corep{U}, 1_{\Hils}
\tenscorep \BrMultunit)\).
Hence the direct sum of countably many copies of~\(\Multunit[C]\)
also comes from a braided representation.  Since
\(\Corep{A}^1\tenscorep\Corep{A}^2 \cong \Corep{A}^1 \tenscorep
1_{\Hils[K]^2} \cong \Corep{A}^1\)
for any representation~\((\Hils[K]^2,\Corep{A}^2)\),
\(\Corep{A}^1\tenscorep\Corep{A}^2\)
also comes from a braided representation.  Now
Lemma~\ref{lem:lift_cancel} shows that any
representation~\((\Hils[K]^2,\Corep{A}^2)\)
of~\(\Multunit[C]\)
comes from a braided representation.  This finishes the proof of
Theorem~\ref{the:braided_to_multunit}.

\end{document}